%% file: main.tex
\documentclass{article}
\usepackage{arxiv}

\usepackage[normalem]{ulem} 

\usepackage[utf8]{inputenc} 
\usepackage[T1]{fontenc}    
\usepackage{hyperref}       
\usepackage{url}            
\usepackage{booktabs}       
\usepackage{amsfonts}       
\usepackage{microtype}      
\usepackage{color}
\usepackage{lipsum}
\usepackage{amsthm}
\usepackage{indentfirst}
\usepackage{graphicx}
\usepackage{epsfig,epsf,latexsym,subfigure}
\usepackage{float}
\usepackage{epstopdf}
\usepackage{fourier}
\usepackage{bm}
\usepackage{bbm} 
\usepackage{mathtools}
\usepackage{latexsym,enumerate}
\usepackage{tabularx}
\usepackage{capt-of}
\usepackage{cases}
\usepackage{enumitem} 
\usepackage{algorithm}
\usepackage{algpseudocode}
\usepackage{fourier}
\usepackage[usenames,dvipsnames]{xcolor}
\usepackage{mathrsfs}

\DeclareMathAlphabet{\mathcal}{OMS}{cmsy}{m}{n}
\SetMathAlphabet{\mathcal}{bold}{OMS}{cmsy}{b}{n}

\def\bm{\boldsymbol}

\newcommand{\comment}[1]{}

\newcommand{\BEA}{\begin{eqnarray}}
\newcommand{\EEA}{\end{eqnarray}}

\newcommand\bR{\mathbb{R}}
\newcommand\cU{\mathcal{U}}
\newcommand\cX{\mathcal{X}}
\newcommand\ddf[2]{\frac{\mathrm{d} #1}{\mathrm{d} #2}}
\newcommand\ppf[2]{\frac{\partial #1}{\partial #2}}

\newtheorem{lemma}{Lemma}[section]
\newtheorem{theorem}{Theorem}[section]
\newtheorem{proposition}{Proposition}[section]
\newtheorem{definition}{Definition}[section]

\newtheorem{assumption}{Assumption}[section]
\newtheorem{remark}{Remark}
\newtheorem{example}{Example}[section]

\title{A model-free method for discovering symmetry in differential equations}

\author{
  Max Kreider \\
  Department of Mathematics \\
  The Pennsylvania State University, University Park, PA 16802, USA\\
  \texttt{mbk6295@psu.edu} \\
  \And
 John Harlim \\
  Department of Mathematics, Institute for Computational and Data Sciences \\
  The Pennsylvania State University, University Park, PA 16802, USA\\
  \texttt{jharlim@psu.edu} \\
  \And
 Daning Huang\\
  Department of Aerospace Engineering \\
  The Pennsylvania State University, University Park, PA 16802, USA\\
  \texttt{daning@psu.edu}
}

\begin{document}

\maketitle

\begin{abstract}

\input{arxiv_sec0_abstract}

\end{abstract}

\section{Introduction}

\input{arxiv_sec1_introduction}

\section{Lie group theory for symmetry analysis}\label{sec2}

\input{arxiv_sec2_liegrouptheory}

\section{Numerical approximation}\label{sec3}

\input{arxiv_sec3_numericalapproximations}

\section{Error analysis}\label{sec4}

\input{arxiv_sec4_erroranalysis}
    

\section{Numerical Experiments}\label{sec5}


\input{arxiv_sec5_results}


\section{Summary and discussion}\label{sec6}

\input{arxiv_sec6_discussion}

\section*{Acknowledgment}

This work is partially supported under the NSF grant DMS-2505605 and the ICDS Penn State seed grant. The research of JH is partially supported by the ONR grant N000142212193.

\appendix

\input{arxiv_sec7_appendices}











\bibliographystyle{plain}  
\bibliography{arxiv}

\end{document}

%% file: arxiv_sec0_abstract.tex
Symmetry in differential equations reveals invariances and offers a powerful means to reduce model complexity. 
Lie group analysis characterizes these symmetries through infinitesimal generators, which provide a local, linear criterion for invariance. 
However, identifying Lie symmetries directly from scattered data, without explicit knowledge of the governing equations, remains a significant challenge. 
This work introduces a numerical scheme that approximates infinitesimal generators from data sampled on an unknown smooth manifold, enabling the recovery of continuous symmetries without requiring the analytical form of the differential equations. 
We employ a manifold learning technique, Generalized Moving Least Squares, to prolongate the data, from which a linear system is constructed whose null space encodes the infinitesimal generators representing the symmetries. 
Convergence bounds for the proposed approach are derived. 
Several numerical experiments, including ordinary and partial differential equations, demonstrate the method’s accuracy, robustness, and convergence, highlighting its potential for data-driven discovery of symmetries in dynamical systems.

%% file: arxiv_sec1_introduction.tex
A wide range of systems possess inherent symmetries, such as invariance under scaling, translation, or rotation.
These symmetries often encode fundamental physical or conservation laws which elucidate system behavior, and can reduce model complexity.
Examples arise across diverse fields, including particle physics \cite{costa2012symmetries, sundermeyer2014symmetries}, gravitational physics \cite{hall2004symmetries}, cosmology \cite{tsamparlis2018symmetries}, finance \cite{gazizov1998lie}, fluid dynamics \cite{cantwell2002introduction,crawford1991symmetry, grundland1995lie,holm1976symmetry, ibragimov2011applications}, and turbulence \cite{khujadze2016revisiting, klingenberg2020symmetries, she2017quantifying}.

This work analyzes Lie point symmetries (LPS), continuous transformation groups acting on the space of independent and dependent variables.
The LPS of differential equations (DEs) are characterized through Lie group theory, in which infinitesimal generators encode nonlinear symmetry transformations through a local, linearized invariance condition.
Lie theory provides a systematic framework for solving DEs, reducing the order of underlying dynamical systems, identifying conserved quantities, and generating new families of solutions from known solutions \cite{bluman2013symmetries,olver1993applications,ovsiannikov2014group}.

However, in modern applications, classical Lie group analysis may be infeasible because systems are often governed by highly complex DEs whose corresponding Lie groups may be analytically intractable, even with the aid of computer algebra systems.
Moreover, in experimental settings, the governing DEs are often unknown and one typically has access to DE solutions only in the form of scattered data.
These limitations motivate the development of data-driven approaches for discovering symmetries in DEs without requiring explicit knowledge of the underlying dynamical equations.

Data-driven methods have been developed for symmetry discovery in data distributions (e.g., a dataset of images) and algebraic functions \cite{dehmamy2021automatic,finzi2021practical,moskalev2022liegg,shaw2024symmetry}.  
Many works which address symmetry in DEs are limited to a particular class of DEs, such as Hamiltonian dynamics \cite{greydanus2019hamiltonian,hou2024machine,kaiser2018discovering,liu2022machine} or ordinary differential equations \cite{yang2023latent,yang2023generative}.  
Recent noteworthy studies address symmetry discovery in partial differential equations (PDEs) and general settings \cite{hu2025explicit,ko2024learning}.
In \cite{ko2024learning}, a neural network (NN) is used to model the infinitesimal generators of a Lie group to represent the LPS.  
While effective,
this approach lacks interpretability and relies on knowledge of the governing DEs. 
In contrast, \cite{hu2025explicit} improves interpretability by representing the infinitesimal generators in terms of  nonlinear features, such as polynomials. 
However, \cite{hu2025explicit} requires training an NN as a surrogate of the DEs, which lacks convergence guarantees and can be difficult to implement.
Furthermore, both \cite{hu2025explicit} and \cite{ko2024learning} use standard finite differences to estimate derivatives, which is impractical for non-uniform experimental data.

In this work, we develop a data-driven approach for identifying LPS directly from data, without requiring knowledge of the governing DEs. 
The method use a manifold learning technique, Generalized Moving Least Squares, to numerically prolongate the given data to its jet space. 
These estimated data are then substituted into a discretized infinitesimal invariance condition, producing a linear system whose null space encodes the infinitesimal generators of the underlying symmetry group.
While a similar strategy was used in \cite{hu2025explicit}, our work eliminates the need for expensive surrogate model training by introducing a novel reformulation of the infinitesimal invariance condition.  
By leveraging established results from manifold learning techniques, we also provide a rigorous convergence guarantee for the proposed algorithm.

The remainder of the manuscript is organized as follows.
In \S\ref{sec2}, we briefly review existing results in Lie group theory for symmetry analysis.
In \S\ref{sec3}, we introduce our numerical algorithm, and provide implementation details and pseudocode.
Error analysis and convergence results are presented in \S\ref{sec4}, followed by a range of numerical examples in \S\ref{sec5}. We conclude with a brief discussion in \S\ref{sec6}. Additional materials, including an overview of Lie group theory, symmetry in algebraic equations, supplementary computational details, GMLS pseudocode, and a detailed analysis of computational complexity, are provided in the Appendices.

%% file: arxiv_sec2_liegrouptheory.tex
Here, we briefly review classical symmetry analysis using Lie group theory; see \cite{olver1993applications} for a more detailed discussion.

\subsection{Lie group theory}

Let $G$ be a $r$-parameter local Lie group (see Appendix \ref{Appendix: Lie group theory definitions}). 
Let $\Phi: \Omega \to M$ be a smooth map corresponding to a local group of transformations on $M$, defined as $(g,x) \mapsto \Phi(g, x)$ for all $g\in G$ and $x\in M$, where $\{0\}\times M \subset \Omega\subset G\times M$. For $(g,x)\in\Omega$, the \textit{infinitesimal generator} of a Lie group $G$ is
 $X_g(x) = \left.\ddf{}{t}\right\vert_{t=0} \Phi(tg,x).$
Since $G$ is locally parameterized by $r$ parameters, we write $g=(g_1,\dots,g_r)$ to specify a particular direction along which the one-parameter subgroup $\Phi(tg,\cdot)$ acts on $M$.
Suppose that $M$ has a coordinate chart $(\mathcal{N},\varphi)$ with $\varphi:\mathcal{N}\to \bR^n$ and local coordinates $\varphi(x)=(x_1,x_2,\cdots,x_n)$. 
We write $\Phi=(\Phi_1,\Phi_2,\cdots,\Phi_n)$ to express the local group action $\Phi$ in these coordinates.
Under the assumption that $\Phi$ maps $\{0\}\times \mathcal{N}$ into $\mathcal{N}$, the infinitesimal generator has the coordinate expansion 
\begin{equation}\label{eqn_gen}
X_g(x) = \sum_{i=1}^n \xi_i(x) \partial_{x_i}\vert_x,\qquad \xi_i(x) = \left.\ppf{\Phi_i}{t}\right\vert_{t=0}.
\end{equation}
The notation $|_x$ indicates that each point $x$ has a distinct tangent space, and will be dropped in the remainder of the manuscript.
It follows that $X_g$ defines a vector field on $M$ in coordinates $x$.
Conversely, solving $\dot{x}_i = \xi_i(x)$, $i=1,2,\dots,n$, for $x\in M$ with $x(0)=x_0$ at $t$ gives a group action on $x_0$, i.e.,
$x(t) = \Phi(tg,x_0)$.
Therefore, the infinitesimal generator $X_g$ is an equivalent representation of the group action. 

\begin{example}\label{exp_so2}
Consider the Lie group $SO(2)$, which represents rotations in the plane.
Since $SO(2)$ is one-dimensional, we can write $g=1\in \mathbb{R}$ to represent the single group parameter direction without loss of generality.
We define $t=\theta$ to match the standard notation for the rotation angle in $\mathbb{R}^2$.
Then, for a point $x=(x_1,x_2)\in M=\bR^2$, the group action is $\Phi(\theta, x) = \left(\cos(\theta) x_1 -\sin(\theta) x_2,\; \sin(\theta) x_1 + \cos(\theta) x_2\right),$
which represents counter-clockwise rotation around the origin point by $\theta$.
The identity transformation corresponds to $\theta=0$.    
Following \eqref{eqn_gen}, the infinitesimal generator in coordinates $(x_1,x_2)$ is $X_g(x) = -x_2\partial_{x_1} + x_1\partial_{x_2}$.
The associated differential equation is $\dot{x}_1 = -x_2,\ \dot{x}_2 = x_1$, whose solutions are circles representing trajectories of rotation.  
Thus, the infinitesimal generator $X_g$ reproduces the $SO(2)$ group action.
\end{example}

\subsection{Symmetry in differential equations}\label{subsection: symmetry in DE}

We now review the classical symmetry analysis of DEs as an extension to the symmetries of algebraic equations (see Appendix \ref{Appendix: symmetry algebraic equations}).
We denote a system of differential equations as
\begin{equation}\label{eqn_de}
    \Delta_\nu(x,u,u^{(p)}) = 0,\quad \nu=1,2,\cdots,q,
\end{equation}
where $x\in\cX\subset\bR^n$ are independent variables, $u\in\cU\subset\bR^m$ are dependent variables, and $u^{(p)}$ are partial derivatives of $u$ up to $p$th order. 
For ease of notation, we define 
\begin{equation}
    z^{(k)} \equiv (x,u,u^{(k)}), \quad k=0,1,\dots,p.\notag
\end{equation}
A $r$th-order derivative of $u_i$ is denoted as
$u_{i,J} = \frac{\partial^r u_i}{\partial x_1^{j_1}\cdots\partial x_n^{j_n}}$,
using a multi-index $J=(j_1,j_2,\cdots,j_n)$, where $|J|=j_1+j_2+\cdots+j_n=r$ and $j_k\geq 0$ for $k=1,2,\cdots,n$.
For example, if $x=(s,t)\in \mathbb{R}^2$ and $u(x)\in\bR$, 
$u^{(1)}=\{u_s,u_t\}$ and $u^{(2)}=u^{(1)}\cup\{u_{ss}, u_{st}, u_{tt}\}$.  
In our notation, a PDE $u_t=u_{ss}$ is written as $\Delta(s,t,u;\ u_s,u_t;\ u_{ss},u_{st},u_{tt}) = u_t-u_{ss} = 0$.

The key to symmetry analysis for DEs is treating the partial derivatives $u^{(p)}$ as independent variables. 
Since the corresponding infinitesimal generators are defined only for the variables $(x,u)$, it is necessary to extend the group action to $u^{(p)}$. 
The extension of variables and group actions is referred to as \textit{prolongation}.
In the prolongated space, the study of symmetries of DEs mirrors the algebraic case (compare Theorem \ref{Theorem: algebraic symmetry} with Theorem \ref{theorem: DE symmetry} below).
We now review prolongation for DEs \cite{olver1993applications}. 

We first consider the prolongation of the variables in the original space $(x,u)\in M =\cX\times\cU \equiv M^{(0)}$.  
Let $u^{(p)} \in \mathcal{U}^{(p)}$, where $\mathcal{U}^{(p)}$ denotes the set of all partial derivatives of $u$ up to order $p$, treated as independent coordinates in a suitable Euclidean space.
The $p$th prolongated space, or \textit{jet space}, is defined as $M^{(p)}\equiv \cX\times\cU\times\cU^{(p)}$.
Note that the system \eqref{eqn_de} defines a \textit{solution manifold} $\mathscr{S}^{(p)}\subset M^{(p)}$, consisting of points $(x,u,u^{(p)})$ corresponding to a connected, smooth subset of the solution space of \eqref{eqn_de}.  
The precise subset should be determined on a case-by-case basis dependent on the PDE and the goals of the analysis.  
Moreover, for each $k=0,1,\dots,p$, we define the $k$th prolongated solution manifold
\begin{equation}
    \mathscr{S}^{(k)} = \bigl\{ (x,u,u^{(k)}) \in M^{(k)} : (x,u) \in \mathscr{S}^{(0)} \bigr\}.\notag
\end{equation}
We assume that $\mathscr{S}^{(p)}$ is a $\mathcal{C}^{\ell+1}$ submanifold of $M^{(p)}$.
Notice that the dimension, $D_k$, of $M^{(k)}$ depends on $k$:
\begin{equation}
    D_k = \text{dim}\, M^{(k)} =  n + m + \sum_{r=1}^k \binom{n+r-1}{r}m.\notag
\end{equation}
However, the intrinsic dimension, $d$, of $\mathscr{S}^{(k)}$ is independent of $k$: $d = \text{dim}\, \mathscr{S}^{(k)}$, $k=0,1,\dots,p$.
The intrinsic dimension of the solution manifold is preserved under prolongation because the coordinates $u^{(k)}$ are uniquely determined by the values of $(x,u)$ along $\mathscr{S}^{(0)}$.

Next, we consider the prolongation of the infinitesimal generators.  
We write the infinitesimal generator at the point $z^{(0)}\in \mathscr{S}^{(0)}\subset M^{(0)}$ in the original coordinates as
\begin{equation}
    X_g = \sum_{i=1}^{n} \xi_i(z^{(0)}) \partial_{x_i} + \sum_{i=1}^{m} \eta_i(z^{(0)}) \partial_{u_i} \equiv X_g(z^{(0)})\cdot\nabla.\notag
\end{equation}
The $p$th \textit{prolongated generator} at the point $z^{(p)} \in \mathscr{S}^{(p)}$ in prolongated coordinates is 
\begin{equation}\label{eqn_pro}
X_g^{(p)} = L^{(p)}X_g \equiv  X_g + \sum_{i=1}^{m} \sum_{|J|=1}^p \eta_{i,J}(z^{(p)}) \partial_{u_{i,J}} \equiv X_g^{(p)}(z^{(p)})\cdot\nabla^{(p)},
\end{equation}
We assume the functions $\xi_i,\eta_i$ are $\mathcal{C}^{\ell+1}$ functions with $\ell\geq p$.
Here, $L^{(p)}: TM^{(0)} \to TM^{(p)}$ is the prolongation operator, mapping a vector field on the tangent bundle $TM$ of $M$ to a vector field on the tangent bundle $TM^{(p)}$ of the prolongated space $M^{(p)}$. 
It acts on $X_g$ to produce its $p$th prolongation.
The functions $\eta_{i,J}$ are given by
\begin{equation}\label{eq: prolongation formula}
    \eta_{i,J} = D_J\left( \eta_i - \sum_{k=1}^n\xi_k\ppf{u_i}{x_k} \right) + \sum_{k=1}^n \xi_k\ppf{u_{i,J}}{x_k},
\end{equation}
where $D_J$ is a composition of total derivative operators (see \cite{olver1993applications}, \S2.3 for more details)
\begin{equation}
    D_J=D_{j_1}D_{j_2}\cdots D_{j_n}, \quad \text{with}\ D_i = \ppf{}{x_i}+\sum_{j=1}^m \sum_{|J|=1}^p \ppf{u_{j,J}}{x_i}\ppf{}{u_{j,J}}.\notag
\end{equation}
The following theorem provides a systematic method to study DE symmetry groups.
\begin{theorem}[Theorem 2.31 in \cite{olver1993applications}]\label{theorem: DE symmetry}
    Consider a system of differential equations of the form \eqref{eqn_de} with maximal rank defined over $M\subset \cX\times \cU$.
    If $G$ is a local group of transformations acting on $M$, and 
    \begin{equation}
        X_g^{(p)}[\Delta_\nu](z^{(p)}) = 0, \quad \nu=1,\dots,q, \quad \text{ whenever } \Delta_\nu(z^{(p)})=0, \text{  i.e.,~} z^{(p)}\in \mathscr{S}^{(p)},\notag
    \end{equation}
    for every infinitesimal generator $X_g$ of $G$, then $G$ is a symmetry group of the system. 
\end{theorem}

We now illustrate Theorem \ref{theorem: DE symmetry} with an example, which continues Example \ref{exp_so2}.
\begin{example}\label{exp_de}
Consider the system of differential equations $\Delta_\nu = 0$, $\nu=1,2$,
\begin{equation}\label{eq: SL oscillator}
    \begin{split}
        \Delta_1(t,x,y,x_t,y_t) &= -x(x^2+y^2-1)+y - x_t = 0,
        \\
        \Delta_2(t,x,y,x_t,y_t) &= -y(x^2+y^2-1)-x - y_t = 0.
    \end{split}
\end{equation}
Here, we will take it as a fact that \eqref{eq: SL oscillator} has the following independent infinitesimal generators
\begin{equation}
        X_{g,1}(t,x,y) =  \partial_t, \quad X_{g,2}(t,x,y) = (- y) \partial_x + ( x) \partial_y.\notag
\end{equation}
We remark that $X_{g,1}$ corresponds to a trivial time-translation symmetry 
and $X_{g,2}$ represents an $SO(2)$ symmetry in solutions of ODEs. Using \eqref{eqn_pro}, one can show that the 1st prolongation of these generators is
\begin{equation}
        X_{g,1}^{(1)}(t,x,y,x_t,y_t) = X_{g,1} + 0, \quad X_{g,2}^{(1)}(t,x,y,x_t,y_t) = X_{g,2} +  (- y_t) \partial_{x_t} + (x_t) \partial_{y_t} .
        \notag
\end{equation}
Note that the Jacobian Matrix, $D\Delta_\nu$, given by
\begin{equation}
    D\Delta_\nu = \begin{bmatrix}
        0 & -3x^2-y^2+1 & -2xy+1 & -1 & 0
        \\
        0 & -2xy-1 & -3y^2 - x^2 + 1 & 0 & -1
    \end{bmatrix},\notag
\end{equation}
is clearly of maximal rank.
One can show that $X_{g,i}^{(1)}[\Delta_\nu] = 0$ for $i=1,2$ and $\nu=1,2$ whenever $\Delta_1(t,x,y,x_t,y_t)=\Delta_2(t,x,y,x_t,y_t) = 0$.
Then, invoking Theorem \ref{theorem: DE symmetry}, the equation $\Delta_\nu(t,x,y,x_t,y_t)=0$ has $SO(2)$ rotational symmetry.
Indeed, one can show that \eqref{eq: SL oscillator} reduces to $r' = -r(r^2-1)$, $\theta'=1$ in polar coordinates.
\end{example}

\subsection{Discovering symmetries in differential equations}\label{subsec: discovering symmetries in DEs}

The previous section (see Theorem \ref{theorem: DE symmetry}) focused on verifying whether a given Lie group is a symmetry group of a DE, assuming the group and its infinitesimal generators are known. 
Conversely, given a DE, it is possible to recover the generators of all possible symmetry groups by formulating and solving a PDE.
Indeed, an infinitesimal generator $X_g$ is related to its prolongated version by the prolongation operator $L^{(p)}$ (recall equation \eqref{eqn_pro}).
The $G$-invariance condition for the solution set of the DE is
\begin{equation}\label{eqn_sym}
X_g^{(p)}[\Delta_\nu]\equiv (X_g^{(p)}\cdot\nabla^{(p)})[\Delta_\nu] = L^{(p)}X_g\cdot\nabla^{(p)}[\Delta_\nu] = 0,
\end{equation}
whenever $\Delta_\nu = 0, \nu=1,\cdots,q$, which produces a system of $q$ PDEs as the governing equations of the generators for the symmetry groups of the DE $\Delta_\nu=0$.  
We refer to \eqref{eqn_sym} as the \textit{(infinitesimal) invariance condition} for a system of DEs $\Delta_\nu=0$.
If \eqref{eqn_sym} does not admit a solution $X_g$, then the DE does not have symmetry in the sense discussed in this study.

While \eqref{eqn_sym} is linear in the infinitesimal generator $X_g$, its components $\xi_i$ and $\eta_i$ can be highly nonlinear functions of $(x,u)$. 
Hence, \eqref{eqn_sym} can be challenging to solve analytically.
A typical approach is to express $\xi_i$ and $\eta_i$ in a finite-dimensional basis
\begin{equation}\label{eq: ansatz}
\begin{split}
    \xi_i(z^{(0)}) &= \sum_{j=1}^\kappa \psi_{j}(z^{(0)})c_{ij} \equiv \psi(z^{(0)})^\top c_i,\quad i=1,2,\cdots,n,
\\
\eta_i(z^{(0)}) &= \sum_{j=1}^\kappa\psi_{j}(z^{(0)})c_{i+n,j} \equiv \psi(z^{(0)})^\top c_{i+n},\quad i=1,2,\cdots,m,
\end{split}
\end{equation}
where $\psi_j(z^{(0)})$ are chosen basis functions and the $c_{ij}$ are unknown coefficients. 
Collecting all coefficients in a vector $\mathbf{c}$, the generator can be written compactly as
\begin{equation}
    X_g = \sum_{i=1}^{n} \xi_i(z^{(0)}) \partial_{x_i} + \sum_{i=1}^{m} \eta_i(z^{(0)}) \partial_{u_i} \equiv \left(\Psi(z^{(0)})\mathbf{c}\right)\cdot \nabla,\notag
\end{equation}
where we introduce the notation
\begin{equation}\label{eq: Psi}
\begin{split}
    \Psi(z^{(0)}) &= \begin{bmatrix}
    \psi(z^{(0)}) & & & \\
     & \ddots & & \\
     & & \psi(z^{(0)})
\end{bmatrix}\in\bR^{(n+m)\times K},
\end{split}
\end{equation}
and $\mathbf{c} = [c_1^\top,\cdots,c_n^\top,c_{n+1}^\top,\cdots,c_{n+m}^\top]^\top \in \bR^{K\times 1}$. Here, $K = \kappa(n+m)$ is the number of unknown coefficients.
This procedure reduces \eqref{eqn_sym} to a finite-dimensional linear system for the coefficients $c_j$, $j=1,2,\dots,n+m$,
\begin{equation}\label{eq: matrix P}
(L^{(p)}\Psi \mathbf{c}\cdot\nabla^{(p)})[\Delta_\nu] = \mathbf{c}^\top \Big((L^{(p)}\Psi)^\top(\nabla^{(p)}[\Delta_\nu])\Big) \equiv \mathbf{c}^\top P_\nu = 0, \quad P\equiv \begin{bmatrix}
    P_1 & \dots & P_q
\end{bmatrix}.
\end{equation}
Note that $L^{(p)}\Psi$ is a $D_p\times K$ matrix-valued function and  $\nabla^{(p)}[\Delta_\nu] \in \mathbb{R}^{D_p\times q}$ has as its $j$th column the gradient $\nabla^{(p)}[\Delta_j]$.
Consequently, $P$ is a ${K\times q}$ matrix-valued function.

\begin{remark}\label{Remark: inf dim symmetry} 
Some differential equations have an infinite-dimensional family of solutions parameterized by one or more arbitrary functions. 
For example, the transport equation $u_t + u_x = 0$ admits the general solution $u = f(x - t)$, where $f$ is an arbitrary smooth function. 

It is important to note that the finite-dimensional formulation \eqref{eq: matrix P} only captures point symmetries that can be expressed within the finite ansatz \eqref{eq: ansatz}, and so our approach excludes infinite-dimensional families of symmetries.
Recovering functional symmetries is beyond the scope of this work.
More details can be found in \cite{olver1993applications}.
\end{remark}

It follows that the basis of the nullspace of the operator $P$, i.e., the set of $c$ such that $c^TP(z^{(p)})=0$ for all $\nu$ and for all $z^{(p)}\in \mathscr{S}^{(p)}$, recovers a finite-dimensional subset of the generators of the symmetry group. 
In this work, we assume that the ansatz \eqref{eq: ansatz} is sufficiently robust to capture symmetries of interest; see \S\ref{sec4} for a discussion of possible error due to this assumption. 
Appendix \ref{Appendix: example} contains a detailed example clarifying the notation used in this section.

\begin{example}\label{ex: SL}
Here, we revisit Example \ref{exp_de}, but suppose that the symmetry group is unknown. 
We adopt a linear ansatz for the infinitesimal generator
\BEA
    X_g &=& \xi_1(t,x,y)\partial_t + \eta_1(t,x,y)\partial_x + \eta_2(t,x,y)\partial_y  \notag
    \\
    &\equiv& (c_0+c_1 t+c_2x + c_3y)\partial_t + (c_4+c_5t+c_6 x + c_7 y)\partial_x + (c_8+c_9t+c_{10} x + c_{11} y)\partial_y.
\notag
\EEA
The form of the prolongated generator is completely determined by the generator in original coordinates.
Using \eqref{eqn_pro}, one sees that
\begin{equation}
    X_g^{(1)} = X_g + \eta_{1,1}(t,x,y,x_t,y_t)\ppf{}{x_t} + \eta_{2,1}(t,x,y,x_t,y_t)\ppf{}{y_t}, \notag
\end{equation}
with
\BEA     \eta_{1,1} &=& (\eta_1)_t + (\eta_1)_x x_t + (\eta_1)_y y_t - (\xi_1)_t x_t - (\xi_1)_x (x_t)^2 - (\xi_1)_y x_t y_t \notag\\ 
     &=& c_5 + c_6 x_t + c_7 y_t - c_1 x_t - c_2 (x_t)^2 - c_3 x_t y_t,\notag
    \\
    \eta_{2,1} &=& (\eta_2)_t + (\eta_2)_x x_t + (\eta_2)_y y_t - (\xi_1)_t y_t - (\xi_1)_y (y_t)^2 - (\xi_1)_x x_t y_t \notag\\
    &=& c_9 + c_{10} x_t + c_{11} y_t - c_1 y_t - c_3 (y_t)^2 - c_2 x_t y_t. \notag
\EEA
Subsequently, consider $X_g^{(1)}[\Delta_\nu]=0$ with the constraint $\Delta_1=\Delta_2=0$ to obtain a linear system in the $c_i$, $i=1,2,\dots,11$.
One looks for $c_i$ that satisfy this system for any $(t,x,y)$, giving $c_1=c_2=c_3=c_4=c_5=c_6=c_8=c_9=c_{11} = 0$, and $-c_7 = c_{10}$.
This computation reveals two degrees of freedom: $c_0$ is free, and $-c_7 = c_{10}$.
Setting $c_0=-c_7 = 1$, one recovers $X_{g,1}^{(1)}(t,x,y,x_t,y_t) = \partial_t$ and $X_{g,2}^{(1)}(t,x,y,x_t,y_t) = (- y) \partial_x + ( x) \partial_y +  (- y_t) \partial_{x_t} + ( x_t) \partial_{y_t}$ as in Example \ref{exp_de}.
\end{example}

%% file: arxiv_sec3_numericalapproximations.tex
The previous section reviewed the classical theory for studying the symmetry groups of differential equations when the underlying dynamical system is explicitly known. 
In this section, we introduce a numerical framework to approximate the symmetry groups of a differential equation when only scattered solution data are available (see Fig.~\ref{fig:sym}). 
We outline the core idea of our algorithm, and subsequently provide implementation details and pseudocode.
A detailed error analysis is given in \S\ref{sec4}.

Suppose we are given data points $\mathbf{z}_i^{(0)}=\{(\mathbf{x}_i,\mathbf{u}_i)\}_{i=1}^N \in \mathscr{S}^{(0)} \subset \mathcal{X}\times \mathcal{U} \subset  \mathbb{R}^{n+m}$ lying on the zero set of $\Delta_\nu$. 
We now formulate a discrete version of the invariance condition \eqref{eq: matrix P} using the available solution data.
Direct evaluation of $P$ is not possible from the data alone, since both $L^{(p)}\Psi$ and $\nabla^{(p)}[\Delta_\nu]$ depend on the prolongated coordinates $u^{(p)}$ in jet space. 
Our approach is to approximate these objects numerically.

First, using standard manifold learning techniques, we prolongate the given data $\mathbf{z}_i^{(0)}$ to obtain approximations $\tilde{\mathbf{u}}^{(p)}_i$ to the derivatives $u^{(p)}(\mathbf{x}_i)$ at each sample point.
These estimates are then used to evaluate 
\begin{equation}\label{eq: L}
    \tilde{\mathbf{L}}_i \equiv L^{(p)}\Psi(\mathbf{x}_i,\mathbf{u}_i,\tilde{\mathbf{u}}_i^{(p)}) \in \mathbb{R}^{D_p\times K},
\end{equation}
which approximates the ground truth $L^{(p)}\Psi(\mathbf{x}_i,\mathbf{u}_i,\mathbf{u}_i^{(p)})$ at each sample point. 
Second, rather than reconstructing $\Delta_\nu$ directly, we observe that $\nabla^{(p)}[\Delta_\nu]$ is normal to the zero set of $\Delta_\nu$, which we assume forms a smooth $d$-dimensional submanifold of the jet space $M^{(p)}$.  
This geometric viewpoint allows us to approximate the normals to the data manifold directly, producing estimates 
\begin{equation}\label{eq: S}
    \tilde{\mathbf{S}}_i = \begin{bmatrix}
        \tilde{\mathbf{S}}_{i,1} & \dots & \tilde{\mathbf{S}}_{i,D_p-d}
    \end{bmatrix}\in \mathbb{R}^{D_p\times (D_p- d)}, \quad
    \tilde{\mathbf{S}}_{i,\nu} \equiv  \nabla^{(p)}[\Delta_\nu](\mathbf{x}_i,\mathbf{u}_i,\tilde{\mathbf{u}}_i^{(p)})\in \mathbb{R}^{D_p}.
\end{equation}
Note that we write $D_p - d$ instead of $q$ because the underlying system of equations is unknown in this context. 
Combining these pieces, we construct a pointwise approximation of $P$ at each data point 
\begin{equation}
\tilde{\mathbf{p}}_{i} = (\tilde{\mathbf{L}}_i)^\top \tilde{\mathbf{S}}_{i} \in \mathbb{R}^{K\times (D_p- d)},\quad
\label{eq: stacked matrix}
\tilde{\mathbf{P}} =
\begin{bmatrix}
\tilde{\mathbf{p}}_{1} &
\tilde{\mathbf{p}}_{2} & \dots & \tilde{\mathbf{p}}_{N}
\end{bmatrix}
\in \mathbb{R}^{K\times N(D_p- d)}.
\end{equation}
The numerical nullspace of $\tilde{\mathbf{P}}$  approximates the desired infinitesimal generators. 

To complete the construction of $\tilde{\mathbf{P}}$, we must approximate the prolongated coordinates ${u}^{(p)}$.
These coordinates can be recovered from knowledge of the tangent vectors at the data points in the $k$th jet space $M^{(k)}$, for $k=0,1,\dots,p$.
We adopt an iterative procedure in which, at each level $k$, we estimate the tangent vectors using a combination of the singular value decomposition (SVD) and generalized moving least squares (GMLS) techniques.
This approach produces successively refined approximations of the derivatives, ultimately yielding the prolongated coordinates $\tilde{\mathbf{u}}^{(p)}_i$ necessary for constructing the pointwise approximations $\tilde{\mathbf{L}}_i$ and $\tilde{\mathbf{S}}_i$.

\begin{remark}
    In our numerical framework, we assume that the given data is sampled from a smooth, finite-dimensional subset of the solution manifold 
        $\mathbf{z}_i^{(0)}\in \mathscr{S}^{(0)} \subset M^{(0)}.$
    Our numerical prolongation procedure will embed this subset into the higher-dimensional jet space $M^{(p)}$, yielding approximations $\tilde{\mathbf{z}}_i^{(k)}=(\mathbf{z}_i^{(0)},\tilde{\mathbf{u}}_i^{(k)})$ to the true data points $\mathbf{z}_i^{(k)} = (\mathbf{z}_i^{(0)},\mathbf{u}^{(k)}_i)\in \mathscr{S}^{(k)}\subset M^{(k)}$ for $k=1,2,\dots,p$.
\end{remark}

\subsection{Tangent vector approximation via SVD}

Suppose we are given scattered data $\{\mathbf{y}_i\}_{i=1}^N\subset \mathbb{R}^n$ sampled from a smooth, $d$-dimensional manifold $\mathscr{S}$.
Let $\mathcal{K}_i = \{\mathbf{y}_{i_1},\dots,\mathbf{y}_{i_{\mathfrak{K}}}\}$ be the set of $\mathfrak K$-nearest neighbors of $\mathbf{y}_i$, with the convention that $\mathbf{y}_i = \mathbf{y}_{i_1}$.
It is possible to obtain a local approximation of the tangent space at each data point using a classical SVD-based method \cite{donoho2003hessian, harlim2023radial, tyagi2013tangent, zhang2025geometric, zhang2004principal} which we now review.
First, one constructs a distance matrix $\mathbf{D}_i \in \mathbb{R}^{n \times \mathfrak K}$ whose $j$th column is the difference $\mathbf{y}_{i_j}-\mathbf{y}_i$.
Next, one computes the SVD of the distance matrix, $\mathbf{D}_i = \mathbf{U}_i \mathbf{\Sigma}_i \mathbf{V}_i$.
Provided that $d<\mathfrak K\ll N$, one expects to recover only $d$ non-trivial singular values.
The corresponding left singular vectors, $\tilde{\mathbf{t}}_i^{(j)}\in \mathbb{R}^n$, $j=1,\dots,d$, give an orthonormal basis, $\tilde{\mathbf{T}}_i = \{\tilde{\mathbf{t}}_i^{(1)},\dots,\tilde{\mathbf{t}}_i^{(d)}\}$, for the approximate tangent space $\widetilde{T_{\mathbf{y}_i}\mathscr{S}}$.
The remaining left singular vectors form an orthonormal basis for the normal space, $\tilde{\mathbf{N}}_i = \{\tilde{\mathbf{n}}_i^{(1)},\dots,\tilde{\mathbf{n}}_i^{(n-d)}\}$.
Further discussion of the implementation and error analysis of this first-order local SVD method can be found in \S3 of \cite{harlim2017parameter}, which reports that the tangent-space estimation error scales as $\mathcal{O}(N^{-1/d})$ for a $d$-dimensional manifold. 

\subsection{Tangent vector refinement via GMLS}

When higher accuracy is required, the Generalized Moving Least Squares (GMLS) method can improve the SVD approximation of the local tangent space \cite{gross2020meshfree,jiang2024generalized,liang2013solving,mirzaei2012generalized, zhang2025geometric}.
Consider the $i$th data point $\mathbf{y}_i$ along with its $\mathfrak K$-nearest neighbors $\mathcal{K}_i$, and an SVD approximation of the tangent space, $\tilde{\mathbf{T}}_i$, and normal space, $\tilde{\mathbf{N}}_i$, at $\mathbf{y}_i$.
For convenience, collect the SVD tangent and normal vectors in a matrix $\tilde{\mathbf{Q}}_i = [\tilde{\mathbf{t}}_i^{(1)},\dots,\tilde{\mathbf{t}}_i^{(d)},\tilde{\mathbf{n}}_i^{(1)},\dots,\tilde{\mathbf{n}}_i^{(n-d)}]\in \mathbb{R}^{n \times n}$, and define the distance matrix $\mathbf{D}_i$ as before. 
Project the displacement vectors onto tangent-normal space by computing $\bm{\mathcal{Q}}_i=\tilde{\mathbf{Q}}_i^\top \mathbf{D}_i \in \mathbb{R}^{n \times \mathfrak K}$.
Note that the columns of $\bm{\mathcal{Q}}_i$, $\{\mathbf{q}_w \in \mathbb{R}^n\}_{w=1}^{\mathfrak K}$, provide a discrete sampling of the manifold in local tangent–normal coordinates.
Let $\bm{\tau}_w$ be the first $d$-components of $\mathbf{q}_w$, i.e., the tangent components, which act as approximate local coordinates on the tangent plane.
Let $\mathbf{s}_w$ be the last $n-d$ components of $\mathbf{q}_w$, i.e., the normal components, which describe how the manifold deviates from the tangent plane.

The GMLS method works by iteratively refining the local representation of the manifold so as to more accurately capture its geometry.  
The manifold can be modeled locally as the graph of a smooth function from the tangent space to the normal space.  
To that end, we represent the manifold locally using a vector-valued polynomial $\tilde \pi_i:\widetilde{T_{\mathbf{y}_i}\mathscr{S}} \to \mathbb{R}^{n-d}$ of degree $\ell$,
\begin{equation}
\tilde \pi_i(\bm{\tau}) = 
\begin{bmatrix} \tilde \pi_i^{(1)}(\bm{\tau}) \\ \vdots \\ \tilde \pi_i^{(n-d)}(\bm{\tau}) \end{bmatrix}, 
\quad \tilde \pi_i^{(r)}(\bm{\tau}) = \sum_{j=1}^Y \beta_{i,j}^{(r)} \, \phi_j(\bm{\tau}), 
\quad r = 1,\dots,n-d,\notag
\end{equation}
where $\Phi(\bm{\tau}) = \{\phi_j(\bm{\tau})\}_{j=1}^Y$ is an ordered monomial basis of $\mathbb{P}^{\ell,d}_{\mathbf{y}_i}$, the space of intrinsic polynomials with degree up to $\ell$ in $d$ variables,
    \begin{equation}
        \Phi(\bm{\tau}) = \{ \bm{\tau}^\gamma : |\gamma| \le \ell \}, \quad \bm{\tau}^\gamma = \tau_1^{\gamma_1} \cdots \tau_d^{\gamma_d}.\notag
    \end{equation}
    Here, $Y = \dim(\mathbb{P}^{\ell,d}_{\mathbf{z}_i}) = \binom{\ell+d}{d}$ is the number of monomials, and  $\beta_{i,j}^{(r)} \in \mathbb{R}$ are unknown coefficients.  
Given the tangent coordinates $\bm{\tau}_w \in \mathbb{R}^d$ and normal deviations $\mathbf{s}_w \in \mathbb{R}^{n-d}$ of the $\mathfrak K$-nearest neighbors, we determine the coefficients $\beta_{i,j}^{(r)}$ by solving the least-squares problem
\begin{equation}\label{eq: least squares problem}
\Big\{\beta_{i,j}^{(r)}\Big\}_{j=1,\dots,Y}^{r=1,\dots,n-d} 
= \arg \min_{\beta} \sum_{w=1}^{\mathfrak K} \big\| \mathbf{s}_w - \tilde \pi_i(\bm{\tau}_w) \big\|_2^2.
\end{equation}
We then define a local coordinate chart for the manifold near the point $\mathbf{y}_i$ by considering the embedding map $\tilde \iota_i : \widetilde{T_{\mathbf{y}_i}\mathscr{S}} \to \mathbb{R}^n$
\begin{equation}\label{eq: embedding map}
    \tilde \iota_i(\bm{\tau}) = \mathbf{y}_i + \tilde{\mathbf{T}}_i \bm{\tau} + \tilde{\mathbf{N}}_i \tilde \pi_i(\bm{\tau}),
\end{equation}
which gives a GMLS-based local parameterization of the manifold near $\mathbf{y}_i$.
The GMLS approximation of the tangent vectors near $\mathbf{y}_i$ is obtained by differentiating \eqref{eq: embedding map} 
\begin{equation}\label{eq: GMLS derivative}
   \hat{\mathbf{T}}_i(\bm{\tau}) \equiv D\tilde \iota_i(\bm{\tau}) =  \tilde{\mathbf{T}}_i + \tilde{\mathbf{N}}_i \Big(D\tilde \pi_i(\bm{\tau})\Big),
\end{equation}
where $D\tilde \pi_i(\bm{\tau})\in \mathbb{R}^{(n-d)\times d}$ is the Jacobian of $\tilde \pi_i(\bm{\tau})$.
The approximated tangent space at $\mathbf{y}_i$ is given by $\hat{\mathbf{T}}_i \equiv \hat{\mathbf{T}}_i(0)$.

\begin{remark}\label{remark: embedding}
Let $\iota : \mathbb{R}^d \to \mathbb{R}^n$ denote the true embedding of the manifold such that $\mathbf{y}_i = \iota(\bm{s}_i)$ for some intrinsic coordinates $\mathbf{s}_i\in \mathbb{R}^d$.  
The corresponding tangent frame at $\mathbf{y}_i$ is given by $\mathbf{T}_i = D\iota(\mathbf{s}_i) \in \mathbb{R}^{n\times d}$.
The GMLS approximation $D\tilde\iota_i(0)$ approximates $D\iota(\mathbf{s}_i)$, with pointwise error
$\mathbf{T}_i - \hat{\mathbf{T}}_i$,
defined up to an orthogonal transformation of the tangent frame, with rate $(N^{-1}\log N)^{\ell/d}$ for appropriate assumptions (see Proposition~\ref{proposition: GMLS error} below.)
\end{remark}

The updated GMLS normal directions at $\mathbf{y}_i$ are obtained as a basis for the orthogonal complement of the tangent space $\hat{\mathbf{N}}_i =  \text{null}(\hat{\mathbf{T}}_i^\top).$
The GMLS frame $[\hat{\mathbf{T}}_i \; \hat{\mathbf{N}}_i]$ is a more accurate representation of the initial SVD approximation for $\ell \geq2$.
Indeed, iterating the GMLS algorithm on the frame $[\hat{\mathbf{T}}_i \; \hat{\mathbf{N}}_i]$ produces successively more accurate approximations of the true tangent-normal frame at $\mathbf{y}_i$.
Iteration is terminated when $\|D\tilde \pi_i(0)\|_2 \leq \epsilon$, where $\epsilon$ is some small tolerance, indicating that the deviation in the normal direction of the local approximation of the tangent plane is negligible.
Appendix \ref{Appendix: GMLS alg} contains further implementation details for the GMLS algorithm.
 
\subsection{Approximation of prolongated data points}\label{subsection: prolongated data points}

The previous subsections described an approach to approximate the tangent vectors at scattered data points $\mathbf{y}_i$.
Here, we consider the case where this data decomposes into ambient and prolongated coordinates: $\mathbf{y}_i = \mathbf{z}_i^{(k)} = (\mathbf{x}_i,\mathbf{u}_i,\mathbf{u}_i^{(k)})$.

    In general, the solution $u(x)$ of a differential equation may depend on the independent variables $x \in \mathbb{R}^n$ and on $r \ge 0$ free constants $C \in \mathbb{R}^r$. 
    In the case of ODEs, the constants $C$ arise as integration constants.
    In the case of PDEs, we assume that one can parameterize the initial conditions with a finite number of parameters.
    For example, recall the linear advection equation from Remark \ref{Remark: inf dim symmetry}.
    Here, the initial condition can be any smooth function $f$.
    If the initial condition is restricted to a Gaussian family, $f(y) = \exp(-(y-C_1)^2/(2C_2))$, then $C=(C_1,C_2)$ parameterizes the center and width of the Gaussian, respectively.
    
    Recall that the intrinsic dimension of the solution manifold is $d = \text{dim}\,\mathscr{S}^{(k)}$ at each prolongation level $k$, with $d$ independent of $k$.
    However, $d$ does depend on $r$, i.e., the intrinsic dimension of the manifold changes if one considers a particular solution with fixed integration constants or a family of solutions parameterized by one or more integration constants (Example \ref{ex: particular vs family} below highlights this distinction).

    \begin{itemize}
        \item If all integration constants are fixed, the solution is uniquely determined by $x$, and the intrinsic dimension is $d=n$. 
        In this case, the tangent vectors of the solution manifold can be computed directly with respect to the independent variables $x$.
        \item If the integration constants are not fixed, the solution depends on additional degrees of freedom, and the intrinsic dimension satisfies $d = n+r > n$. 
        In this scenario, the tangent vectors must account for variations along both $x$ and the free constants $C$.
    \end{itemize}

    In our numerical framework, we unify these cases by introducing an \textit{augmented set of independent variables}.
    By abuse of notation we denote 
        $x \equiv (x, C) \in \mathbb{R}^d$,
    which subsumes both the true independent variables and any effective degrees of freedom from free constants.  
    This notation highlights the fact that if only data is available, free integration constants are effectively independent variables and should be treated as such in our numerical approach.
    Therefore, without loss of generality we assume that $n=d$ throughout the rest of the manuscript.

Let $\mathbf{T}_{k,i}^j$ denote the collection of tangent vectors at the true $k$-prolongated data $\mathbf{z}_i^{(k)}$, with local coordinates $\bm{\varsigma}^{(i)} = (\varsigma_{1}^{(i)},\dots,\varsigma_{d}^{(i)})\in \mathbb{R}^{d}$.
Since the intrinsic dimension $d$ is independent of $k$, the same coordinates $\bm{\varsigma}^{(i)}$ can be used for all prolongations. 
Different coordinates $\bm{\varsigma}_k^{(i)}$ could be chosen for each $k$, but for simplicity we adopt a single coordinate system throughout.

Consider the tangent vectors of the ambient space at point $\mathbf{z}_i^{(0)} = (\mathbf{x}_i, \mathbf{u}_i)$, where $\mathbf{x}_i = ((x_i)_1,\dots,(x_i)_d) \in \mathbb{R}^d$ (here, we emphasize that $x$ refers to the augmented set of independent variables and integration constants) and $\mathbf{u}_i = ((u_i)_1,\dots,(u_i)_m) \in \mathbb{R}^{m}$:
\begin{equation}
\mathbf{T}_{0,i}^j
= 
\begin{bmatrix}
\frac{\partial (x_i)_1}{\partial \varsigma_{j}^{(i)}} & \frac{\partial (x_i)_2}{\partial \varsigma_{j}^{(i)}} & \cdots & \frac{\partial (x_i)_d}{\partial \varsigma_{j}^{(i)}} &
\frac{\partial (u_i)_1}{\partial \varsigma_{j}^{(i)}} & \cdots & \frac{\partial (u_i)_{m}}{\partial \varsigma_{j}^{(i)}}
\end{bmatrix} \in \mathbb{R}^{d+m},
\label{eq:tanvector}
\end{equation}
 with $j=1,\dots,d$, where all derivatives are evaluated at $\bm{\varsigma}^{(i)} = \mathbf{0}$.
Then, by the chain rule, the partial derivatives $\partial (u_i)_b / \partial (x_i)_j$ for $b=1,2,\dots,m$ and $j=1,2,\dots,d$ satisfy the linear system
\begin{equation}\label{eq: tangent linear system}
\begin{bmatrix}
\frac{\partial (u_i)_b}{\partial \varsigma_1^{(i)}} \\
\vdots \\
\frac{\partial (u_i)_b}{\partial \varsigma_d^{(i)}}
\end{bmatrix}
=
\begin{bmatrix}
\frac{\partial (x_i)_1}{\partial \varsigma_1^{(i)}} & \dots & \frac{\partial (x_i)_d}{\partial \varsigma_1^{(i)}} \\
\vdots & & \vdots \\
\frac{\partial (x_i)_1}{\partial \varsigma_d^{(i)}} & \dots & \frac{\partial (x_i)_d}{\partial \varsigma_d^{(i)}}
\end{bmatrix}
\begin{bmatrix}
\frac{\partial (u_i)_b}{\partial (x_i)_1} \\
\vdots \\
\frac{\partial (u_i)_b}{\partial (x_i)_d}
\end{bmatrix},
\end{equation}
which is understood to be evaluated at $\bm{\varsigma}^{(i)} = \mathbf{0}$.
We write \eqref{eq: tangent linear system} as $\mathscr{B}_i=\mathscr{A}_i \mathscr{X}_i $, noting that $\mathscr{A}_i$ and $\mathscr{B}_i$ depend only on the known entries of $\mathbf{T}_{0,i}^j$.
    Without loss of generality, assume that $(x_i)_k:\mathbb{R}^d \to \mathbb{R}$ is a regular parameterization defined as $({x}_i)_k(\mathbf{0}) = (\mathbf{x}_i)_k$ for $k=1,\ldots, d$, which ensures that $\mathscr{A}^{-1}$ exists and allows us to write $\mathscr{X}_i = \mathscr{A}_i^{-1}\mathscr{B}_i$.
    This procedure can be repeated for each $b=1,\dots,m$ to recover the prolongated coordinates $\mathbf{z}_i^{(1)} = \{(\mathbf{x}_i,\mathbf{u}_i,\mathbf{u}_i^{(1)})\}$.

In practice, the joint SVD–GMLS approach provides approximations $\hat{\mathbf{T}}_{k,i}^j$ given approximated jet space data $\tilde{\mathbf{z}}_i^{(k-1)}$.
Exactly as above, one construct a linear system $\tilde{\mathscr{B}}_i=\tilde{\mathscr{A}}_i \tilde{\mathscr{X}}_i $ whose solution $\tilde{\mathscr{X}}_i$ may be used to approximate the $k$th jet data $\tilde{\mathbf{z}}_i^{(k)}$.

This idea can be applied recursively to prolongate the given data in the ambient space.
Suppose that the sampled data $\mathbf{z}_i^{(0)}=\{(\mathbf{x}_i,\mathbf{u}_i)\}_{i=1}^N \subset M$ is given.
An initial application of the SVD–GMLS method, as outlined above, yields approximations of the tangent space at each point.
By solving a hierarchy of linear systems, one may compute $\tilde{\mathbf{u}}_i^{(1)}$, an approximation of the first derivatives of $\mathbf{u}_i$, and arrive at the augmented dataset
\begin{equation}
    \tilde{\mathbf{z}}_i^{(1)}=\{(\mathbf{x}_i,\mathbf{u}_i,\tilde{\mathbf{u}}_i^{(1)}) \}_{i=1}^N \approx \mathbf{z}_i^{(1)}\in \mathscr{S}^{(1)}.\notag
\end{equation}
By iterating this procedure on the computed data $\tilde{\mathbf{z}}_i^{(k-1)}$, one obtains approximations $\tilde{\mathbf{z}}_i^{(k)}$ and so on.
After $p$ iterations, one arrives at prolongated data
\begin{equation}
   \tilde{\mathbf{z}}_i^{(p)} =  \{(\mathbf{x}_i,\mathbf{u}_i,\tilde{\mathbf{u}}_i^{(p)}) \}_{i=1}^N\approx \mathbf{z}_i^{(p)}\in \mathscr{S}^{(p)}.\notag
\end{equation}
Once the $p$th-order prolongated data is available, one may compute $\mathbf{\tilde{L}}_i$, an approximation of $L^{(p)}\Psi$ at the $i$th sample point.
A final application of SVD–GMLS to the prolongated dataset $\tilde{\mathbf{z}}_i^{(p)}$ provides tangent and normal vectors in jet space.
These normal vectors are precisely the $\tilde{\mathbf{S}}_{i,\nu}$, which approximate $\nabla^{(p)}[\Delta_\nu]$ at the corresponding sample points.
The last step is to assemble the matrix $\tilde{\mathbf{P}}$ according to \eqref{eq: stacked matrix} and compute its numerical nullspace, which provides an approximation of the Lie algebra of symmetries directly from sampled data.

We conclude this section by summarizing the entire numerical procedure in Algorithm \ref{alg:summary}.
The computational complexity of algorithm \ref{alg:summary} is at best $\mathcal{O}(p D_p N \log N)$ and at worst $\mathcal{O}(p D_p N^2)$, depending on the nearest neighbor search algorithm and the size of the data set.
Further details are provided in Appendix \ref{Appendix: complexity}.

\begin{algorithm}
\caption{Summary of data-driven method}\label{alg:summary}

\begin{algorithmic}[1]
\Require \parbox[t]{\dimexpr.9\linewidth-\algorithmicindent}{%
A set of (distinct) data points $\{(\mathbf{x}_i,\mathbf{u}_i)\}_{i=1}^N$ on a $d$-dimensional manifold $\mathscr{S}^{(0)}$, ansatz matrix $\Psi$ (see \eqref{eq: Psi}), nearest neighbor parameter $\mathfrak K \ll N$, stopping threshold $\epsilon=10^{-12}$, polynomial degree $\ell$.}

\For{$k\in{0,\cdots,p-1}$}
\For{$i\in\{1,\cdots,N\}$}
\State For data $\tilde{\mathbf{z}}_i^{(k)}$, compute initial tangent–normal frames $[\tilde{\mathbf{T}}_i, \tilde{\mathbf{N}}_i]$ via the SVD.
\State Refine these frames $[\hat{\mathbf{T}}_i, \hat{\mathbf{N}}_i]$ using GMLS (see Appendix \ref{Appendix: GMLS alg}).
\State Prolongate with GMLS to obtain approximations $\tilde{\mathbf{z}}_i^{(k+1)}=\{(\mathbf{x}_i,\mathbf{u}_i,\tilde{\mathbf{u}}_i^{(k+1)})\}$.
\EndFor
\EndFor

\State Compute $\tilde{\mathbf{L}}_i$ (see \eqref{eq: L}) for each $i=1,\dots,N$.
\State Repeat steps 3 and 4 on $\tilde{\mathbf{z}}_i^{(p)}$ to compute tangents, and use these to obtain $\tilde{\mathbf{S}}_i$ (see \eqref{eq: S}), for each $i=1,\dots,N$.

\State Assemble the stacked matrix $\tilde{\mathbf{P}}$ according to \eqref{eq: stacked matrix} using $\tilde{\mathbf{L}}_i$, $\tilde{\mathbf{S}}_i$, and $\Psi$.
\State Compute the numerical nullspace of $\tilde{\mathbf{P}}$ via the SVD to approximate the infinitesimal generators.

\Ensure\parbox[t]{\dimexpr.9\linewidth-\algorithmicindent}{Numerical nullspace of $\tilde{\mathbf{P}}$, approximating the infinitesimal generators.}
\end{algorithmic}
\end{algorithm}

\begin{figure}
    \centering
    \includegraphics[width=\linewidth]{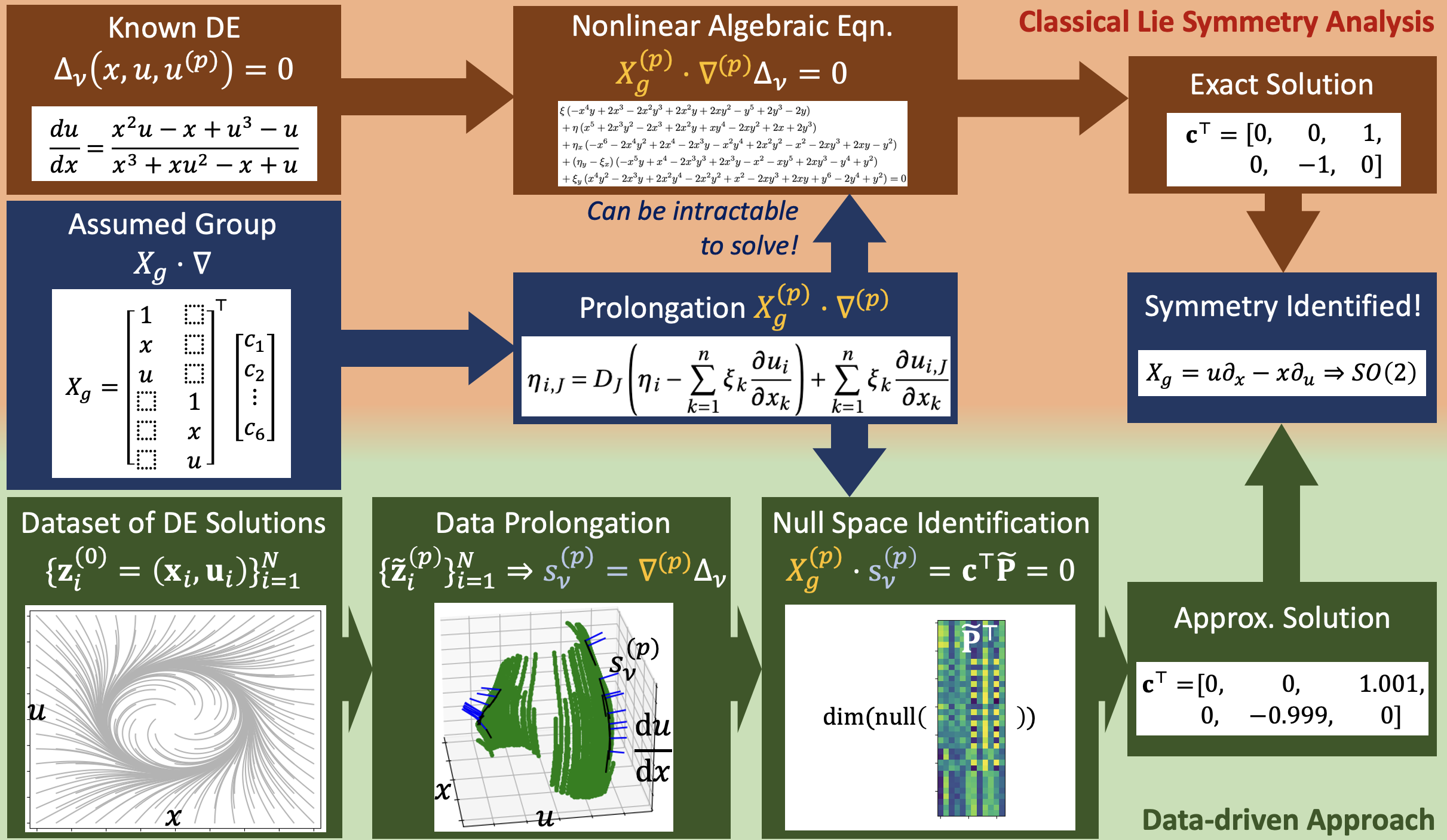}
    \caption{Comparison of analytical and data-driven procedures of symmetry analysis.}
    \label{fig:sym}
\end{figure}

%% file: arxiv_sec4_erroranalysis.tex
In this section, we provide a detailed error analysis of our numerical procedure.  
To fix notation, denote $\mathcal{P}[X_g] \equiv L^{(p)}X_g \cdot \nabla^{(p)}[\Delta_\nu]$. 
We view $\mathcal{P}:TM^{(0)} \to \mathcal{C}^{\ell+1}(M^{(p)},\mathbb{R}^q)$ as an operator that maps vector fields $X_g:M^{(0)}\to TM^{(0)}$ to a $q$-tuple of smooth functions on the $p$th jet space, defined by 
\begin{equation}
    \mathcal{P}[X_g](z^{(p)}) = 
\left( (L^{(p)}X_g)(z^{(p)}) \cdot \nabla^{(p)}[\Delta_j](z^{(p)}) \right)_{j=1,\ldots, q}
\end{equation}
where $z^{(p)} = (x,u,u^{(p)}) \in \mathscr{S}^{(p)} \subset M^{(p)}$ are coordinates in jet space.

The goal of our algorithm is to approximate the nullspace of $\mathcal{P}$,
\begin{equation}
    \text{null } \mathcal{P} = \Big\{ X_g : L^{(p)}X_g \cdot \nabla^{(p)}[\Delta_\nu] = \bm{0} \Big\}, \quad \nu = 1,2,\dots,q.\notag
\end{equation}
Conceptually, two distinct sources of error arise in our approach.  
First, we approximate $\mathcal{P}$ by the object $P$ according to \eqref{eq: matrix P}.
We will find it useful to view $P: (TM^{(0)})_{\text{span}\{\Psi\}} \to \mathcal{C}^{\ell+1}(M^{(p)},\mathbb{R}^{K\times q})$ as an operator which maps a finite-dimensional subspace of vector fields spanned by $\Psi$ (recall the ansatz \eqref{eq: ansatz}) to a $K\times q$ matrix-valued function, defined by
\begin{equation}
    P(z^{(p)}) = \begin{bmatrix}
        (L^{(p)}\Psi(z^{(p)}))^T (\nabla^{(p)}[\Delta_1](z^{(p)})) & \dots & (L^{(p)}\Psi(z^{(p)}))^T (\nabla^{(p)}[\Delta_q](z^{(p)}))
    \end{bmatrix}.\notag
\end{equation}
If the chosen basis is not sufficiently expressive, some or all symmetries of the underlying dynamical system may not be captured by $P$.  
Second, $P$ is not directly available since the governing DE is unknown; instead, our data-driven method produces an approximation $\tilde{\mathbf{P}}$ of $P(z^{(p)})$ for all $z^{(p)}\in\mathscr{S}^{(p)}$, which introduces numerical error.

Ideally, one would compare the nullspaces of $\mathcal{P}$ and $\tilde{\mathbf{P}}$ directly.  
However, such a comparison is problematic in general: for instance, the heat equation $u_t = u_{xx}$ possesses a symmetry group spanned by six vector fields together with an infinite-dimensional subalgebra (see \S 2.4 in \cite{olver1993applications} for a full treatment of the heat equation, and \S\ref{subsection: ex heat equation} for a numerical approximation of a subset of these symmetries).  
Thus, the nullspace of $\mathcal{P}$ cannot, in general, be represented by a finite-dimensional matrix. 
In this work, we focus on the case where the chosen finite-dimensional basis ($\psi$ in \eqref{eq: ansatz}) is sufficiently rich to capture one or more symmetries of interest.  
Under this assumption, the nullspace of $P(z^{(p)})$ encodes one or more (though not necessarily all) symmetries of the underlying system.  
Our goal in the remainder of this section is to study the error inherent in $\tilde{\mathbf{P}}$ to assess the fidelity with which our algorithm recovers these symmetries.
We conclude the section with a brief remark about possible bias error introduced by the choice of ansatz. 

Assume that initial data $\mathbf{z}_i^{(0)}$ is given and that approximate data $\tilde{\mathbf{z}}_i^{(p)}$ has been computed.
We first relate $P(\mathbf{z}_i^{(p)})$ to its numerical approximation $\tilde{\mathbf{P}}$ with an additive error term $\mathbf{E}$.
Notice that the dimensions of the matrix $\tilde{\mathbf{P}}$ and $P(\mathbf{z}_i^{(0)})$ differ because $\tilde{\mathbf{P}}$ concatenates many copies of $P(\mathbf{z}_i^{(0)})$ at each of the $N$ approximated data points.
For that reason, we define 
\begin{equation}
    \mathbf{P} = \begin{bmatrix}
        P(\mathbf{z}_1^{(0)}) & \dots &  P(\mathbf{z}_N^{(0)})
    \end{bmatrix} \in \mathbb{R}^{K\times N(D_p-d)}\notag
\end{equation}
as the matrix consisting of $N$ copies of the object $P$ evaluated at each data point stacked together so that $\mathbf{P}$ and $\tilde{\mathbf{P}}$ have the same dimension.

The data-driven matrix $\tilde{\mathbf{P}}$ is subject to numerical errors, and therefore does not have an exact nullspace—only an approximate one associated with its smallest singular values. 
To assess the fidelity of this numerical approximation, we establish an error bound on the angle between the subspaces spanned by the right singular vectors of $\mathbf{P}$ and $\tilde{\mathbf{P}}$.  
This result is stated in Theorem \ref{theorem: final error analysis} below.
Since the proof is somewhat involved, we first review several relevant auxiliary results and derive several new intermediate results (Lemmas~\ref{proposition: jet data accuracy}-\ref{proposition: P error}) required for the full proof.

We begin by recalling a sine theorem due to Davis and Kahan \cite{davis1970rotation}, which provides a quantitative comparison between the subspaces spanned by singular vectors of a matrix $\mathbf{A}$ and its perturbed counterpart $\mathbf{A} + \bm\Upsilon$, where $\bm\Upsilon$ is a small error term. 
The result relates the principal angles between the corresponding invariant subspaces to the norm of the perturbation.

\begin{theorem}[Davis and Kahan sine theorem, \cite{davis1970rotation}]\label{theorem: sine theorem}
    Suppose that 
    \begin{equation}
        \begin{split}
            \mathbf{A} &= \mathbf{H}_0 \mathbf{A}_0 \mathbf{H}_0^* + \mathbf{H}_1 \mathbf{A}_1 \mathbf{H}_1^* ,\\
\tilde{\mathbf{A}} = \mathbf{A} + \bm{\Upsilon}  &= \mathbf{F}_0 \bm{\Lambda}_0 \mathbf{F}_0^* + \mathbf{F}_1 \bm{\Lambda}_1 \mathbf{F}_1^* ,
        \end{split}\label{eq:def A}
    \end{equation}
    are Hermitian and that $[\mathbf{H}_0 \; \mathbf{H}_1]$ and $[\mathbf{F}_0 \; \mathbf{F}_1]$ are orthogonal sets.
    Define the residual
  $\mathbf{R} = \tilde{\mathbf{A}} \mathbf{H}_0 - \mathbf{H}_0 \mathbf{A}_0$.
    If there exists an interval $[a,b]$ and $\delta > 0$ such that the spectrum of $\mathbf{A}_0$ lies entirely in $[a,b]$ while that of $\bm{\Lambda}_1$ lies entirely outside of $(a-\delta,b+\delta)$, then $\|\sin\bm\Theta\|_2 \leq \delta^{-1}\|\mathbf{R}\|_2$, where $\bm\Theta$ is the diagonal matrix of principal angles between the subspaces spanned by $\mathbf{H}_0$ and $\mathbf{F}_0$.
\end{theorem}

In our context, we define $\mathbf{A}=\mathbf{P} \mathbf{P}^\top$ and $\tilde{\mathbf{A}}=\tilde{\mathbf{P}} \tilde{\mathbf{P}}^\top$.
Note that both $\mathbf{A}$ and $\tilde{\mathbf{A}}$ are Hermitian, and that $\text{null}({\mathbf{P}})=\text{null}({\mathbf{A}})$ and  $\text{null}({\tilde{\mathbf{P}}})=\text{null}({\tilde{\mathbf{A}}})$.
Our strategy is to apply the results of Theorem \ref{theorem: sine theorem} to $\mathbf{A}$ and $\tilde{\mathbf{A}}$.
To do so, our goal is to express $\tilde{\mathbf{A}} = \mathbf{A} + \bm\Upsilon$, and derive a bound on $\|\bm\Upsilon\|_2$.
To that end, we recall a useful error estimate for the GMLS algorithm \cite{jiang2024generalized}.

\begin{proposition}\label{proposition: GMLS error}
Let $\mathbf{z}_i^{(0)}$ be a set of $N$ uniformly sampled i.i.d.~data from a $d$-dimensional manifold $\mathscr{S}$. 
Assume that $g \in C^{\ell+1}(\mathscr{S})$.
Let $\hat g$ be the GMLS approximation to $g$. 
Then, w.p.h. (with probability higher) than $1-\frac{1}{N}$,
\begin{equation}
   \left|D g(\mathbf{z}_i^{(0)}) - D\hat{g}(\mathbf{z}_i^{(0)}) \right|  = \mathcal{O}(\epsilon_d), \quad \epsilon_d =  \left( \left(\frac{\log N}{N} \right)^{\frac{\ell}{d}}\right),\notag
\end{equation}
as $N\to \infty$ for all $\mathbf{z}_i^{(0)}$. 
The constant in the big-$\mathcal{O}$ notation can depend on $d$, but is independent of $N$. 
\end{proposition}

We now utilize Proposition \ref{proposition: GMLS error} to derive an error bound on the accuracy of the prolongated data in jet space from the GMLS procedure, which will later be used to provide an error bound on $\|\bm\Upsilon\|_2$.
To do so, we must consider linear systems of the form \eqref{eq: tangent linear system}.
We will assume the following regularity conditions:

\begin{assumption}\label{assumption}
    Without loss of generality, we take the same local coordinates $\bm{\varsigma}^{(i)}$ for each prolongation level and assume that these parameterizations are regular, i.e., that $(x_i)_j:\mathbb{R}^d \to \mathbb{R}$ is a regular parameterization defined as $({x}_i)_j(\varsigma) = (\mathbf{x}_i)_j$ for $j=1,\ldots, d$.
    This condition ensures that the matrix $\mathscr{A}_i$ is invertible for each $i=1,\dots,N$.
    We further assume that there exist constants $0<a\leq b<\infty$, independent of $d$ and $N$, such that the singular values of $\mathscr{A}_i$ satisfy
    \begin{equation}
        a \leq \sigma_j(\mathscr{A}_i) < b, \quad j=1,2,\dots,d, \notag
    \end{equation}
    for each $i=1,2,\dots,N$, such that         $\|\mathscr{A}_i^{-1}\|_2  = \mathcal{O}(1)$.
\end{assumption}

\begin{lemma}\label{proposition: jet data accuracy}
    Let $\mathbf{z}_i^{(0)}=(\mathbf{x}_i,\mathbf{u}_i)$ be a given set of $N$ i.i.d.~samples drawn uniformly from a $d$-dimensional manifold $\mathscr{S}^{(0)}$, where $\mathbf{x}_i = ((x_i)_1,\dots,(x_i)_d) \in \mathbb{R}^d$ and $\mathbf{u}_i = ((u_i)_1,\dots,(u_i)_m) \in \mathbb{R}^{m}$. Let $\mathscr{S}^{(p)}$ denote the $p$th prolongated solution manifold corresponding to $\mathscr{S}^{(0)}$ be $\mathcal{C}^{\ell+1}$, with $\ell \geq 2$ 
    , and denote $\mathbf{z}_i^{(p)} = (\mathbf{x}_i,\mathbf{u}_i,\mathbf{u}_i^{(p)})\in \mathscr{S}^{(p)}$ as the true $p$th order jet space data associated with $\mathbf{z}_i^{(0)}$.  
    Let $\tilde{\mathbf{z}}_i^{(k)} = (\mathbf{x}_i,\mathbf{u}_i,\tilde{\mathbf{u}}_i^{(k)})$ be the GMLS approximation to $\mathbf{z}_i^{(k)}$ according to the procedure in \S\ref{subsection: prolongated data points}.
    Suppose that Assumption \ref{assumption} holds.
    Then,  {w.p.h. than $1-\frac{k}{N}$,}
    \begin{equation}
        \| \mathbf{z}_i^{(k)} - \tilde{\mathbf{z}}_i^{(k)}\|_2 = \mathcal{O}(d^{3k/2}\epsilon_d), \quad 1\leq i \leq N, \quad 1\leq k \leq p. \notag
    \end{equation}
\end{lemma}

\begin{proof}
    Define the tangent vector function $T^{(k)}(\mathbf{z}_i^{(k)}) = D\iota^{(k)}(\bm{\varsigma}^{(i)})\in \mathbb{R}^{D_k\times d}$ which maps a data point in the $k$th jet space to the corresponding tangent space at that point.
    Here, \(\iota^{(k)}\) is the embedding of the $k$th jet-space manifold \(\mathscr{S}^{(k)}\) into its ambient space $M^{(k)}$.
    Let $\hat{T}^{(k)}$ be the GMLS approximation of the true tangent vector function
    (see Remark \ref{remark: embedding} and preceding discussion).
    Let $\mathbf{T}_{k,i}^j$ for $j=1,\dots,d$ denote the collection of tangent vectors at $\mathbf{z}_i^{(k)}$ with local coordinates $\bm{\varsigma}^{(i)} = (\varsigma_{1}^{(i)},\dots,\varsigma_{d}^{(i)})\in \mathbb{R}^{d}$.
    Let $\hat{\mathbf{T}}_{k,i}^j$ represent the GMLS approximation of these tangent vectors.

    We first consider the case when $k=0$ with given data $\mathbf{z}_i^{(0)}$.
    Explicitly, we express the true tangent vectors as $\mathbf{T}_{0,i}^j$ as defined in \eqref{eq:tanvector}.
\comment{\begin{equation}
\mathbf{T}_{0,i}^j
= 
\begin{bmatrix}
\frac{\partial (x_i)_1}{\partial \varsigma_{j}^{(i)}} & \frac{\partial (x_i)_2}{\partial \varsigma_{j}^{(i)}} & \cdots & \frac{\partial (x_i)_d}{\partial \varsigma_{j}^{(i)}} &
\frac{\partial (u_i)_1}{\partial \varsigma_{j}^{(i)}} & \cdots & \frac{\partial (u_i)_{m}}{\partial \varsigma_{j}^{(i)}}
\end{bmatrix} \in \mathbb{R}^{d+m},
\quad j=1,\dots,d.\notag
\end{equation}}
    By Proposition \ref{proposition: GMLS error}, we conclude that w.p.h.~than $1-\frac{1}{N}$, $\hat{\mathbf{T}}_{0,i}^j =  \mathbf{T}_{0,i}^j + \mathcal{O}(\epsilon_{d})$, in a pointwise sense.

    With an error bound on the tangent vectors established, we next quantify how this approximation error propagates to the prolongated coordinates $\mathbf{z}_i^{(1)}$ obtained from the derivatives represented by these tangent vectors. The chain rule implies that the true derivatives satisfy the linear system $\mathscr{B}_i=\mathscr{A}_i\mathscr{X}_i$ in \eqref{eq: tangent linear system} 
    \comment{
    for $b=1,\dots,m$,
    \begin{equation}\label{eq: tangent linear system v2}
\begin{bmatrix}
\frac{\partial (u_i)_b}{\partial \varsigma_1^{(i)}} \\
\vdots \\
\frac{\partial (u_i)_b}{\partial \varsigma_d^{(i)}}
\end{bmatrix}
=
\begin{bmatrix}
\frac{\partial (x_i)_1}{\partial \varsigma_1^{(i)}} & \dots & \frac{\partial (x_i)_d}{\partial \varsigma_1^{(i)}} \\
\vdots & & \vdots \\
\frac{\partial (x_i)_1}{\partial \varsigma_d^{(i)}} & \dots & \frac{\partial (x_i)_d}{\partial \varsigma_d^{(i)}}
\end{bmatrix}
\begin{bmatrix}
\frac{\partial (u_i)_b}{\partial (x_i)_1} \\
\vdots \\
\frac{\partial (u_i)_b}{\partial (x_i)_d}
\end{bmatrix},
\end{equation}}
which is understood to be evaluated at $\bm{\varsigma}^{(i)} = \mathbf{0}$, where the superscript indexes the $i$th data point.
The GMLS approximation gives a perturbed linear system $\tilde{\mathscr{B}}_i = \tilde{\mathscr{A}}_i \tilde{\mathscr{X}}_i$, where $\tilde{\mathscr{B}}_i = \mathscr{B}_i + \Delta \mathscr{B}_i$ and $\tilde{\mathscr{A}}_i = \mathscr{A}_i + \Delta \mathscr{A}_i$.
The assumption of a regular parameterization ensures that $\mathscr{A}_i$ is invertible.
Expanding the perturbed equation and ignoring doubly infinitesimal terms, we find that
\comment{
    \begin{equation}
        \hat{\mathscr{X}}_i-\mathscr{X}_i = \Delta\mathscr{X}_i = \mathscr{A}_i^{-1}\Big(\Delta\mathscr{B}_i - (\Delta\mathscr{A}_i)\mathscr{X}_i\Big),\notag
    \end{equation}
    which implies the following inequality}
    \begin{equation}\label{eq: abs error bound}
        \|\Delta \mathscr{X}_i\|_2 \leq \|\mathscr{A}_i^{-1}\|_2\left(\\|\Delta\mathscr{B}_i\|_2 + \|\Delta \mathscr{A}_i\|_2\|\mathscr{A}_i^{-1}\|_2\|\mathscr{B}_i\|_2 \right). 
    \end{equation}
\comment{    To establish an asymptotic bound, note that the regularity assumptions imply that
    \begin{equation}\label{eq: inequalities 1}
    \begin{split}
        \|\mathscr{A}_i^{-1}\|_2  &= \mathcal{O}(1).
    \end{split}
    \end{equation}
    Note also the following inequalities, which hold for arbitrary $d$-dimensional vectors and $d\times d$ matrices
    \begin{equation}\label{eq: inequalities 2}
        \begin{split}
            \|\mathscr{B}_i\|_2 \leq \sqrt{d}\|\mathscr{B}_i\|_{\max} &= \mathcal{O}(d^{1/2}),
            \\
            \|\Delta \mathscr{B}_i\|_2 \leq \sqrt{d}\|\Delta \mathscr{B}_i\|_{\max} &= \mathcal{O}(d^{1/2} \epsilon_d),
            \\
            \|\Delta \mathscr{A}_i\|_2 \leq d \|\Delta \mathscr{A}_i\|_{\max} &= \mathcal{O}(d\epsilon_d).
        \end{split}
    \end{equation}}
    By Assumption~\ref{assumption}, the standard norm inequalities, $\|\mathscr{B}_i\|_2 \leq \sqrt{d}\|\mathscr{B}_i\|_{\max} =\mathcal{O}(d^{1/2})$, $\|\Delta \mathscr{B}_i\|_2 \leq \sqrt{d}\|\Delta \mathscr{B}_i\|_{\max} = \mathcal{O}(d^{1/2} \epsilon_d)$, $\|\Delta \mathscr{A}_i\|_2 \leq d \|\Delta \mathscr{A}_i\|_{\max} = \mathcal{O}(d\epsilon_d)$, and that the components of $\tilde{\mathscr{X}}_i$ are the approximations $\tilde{\mathbf{u}}_i^{(1)}$, we have w.p.h.~than $1-1/N$, 
    \begin{equation}
    \|\tilde{\mathbf{z}}_i^{(1)} -  \mathbf{z}_i^{(1)} \|_2 =\|\Delta\mathscr{X}_i\|_2= \mathcal{O}(d^{3/2}\epsilon_{d}).\notag
    \end{equation}
    
    We now have data $\tilde{\mathbf{z}}_i^{(1)} = \{(\mathbf{x}_i,\mathbf{u}_i,\tilde{\mathbf{u}}_i^{(1)})\}$.
    Arguments similar to the above show that passing from tangent vector to derivative retains the order of accuracy with an additional multiplicative factor of $d^{3/2}$.
    Consequently, it suffices to consider the GMLS error when approximating the tangent vectors at the \textit{approximated} prolongated data points $\tilde{\mathbf{z}}_i^{(1)}$.
    We consider the difference
    \begin{equation}
            \|\hat T^{(1)}(\tilde{\mathbf{z}}_i^{(1)}) - T^{(1)}(\mathbf{z}_i^{(1)})\|_2 \leq \|\hat T^{(1)}(\tilde{\mathbf{z}}_i^{(1)}) - \hat T^{(1)}(\mathbf{z}_i^{(1)})\|_2 + \|\hat T^{(1)}({\mathbf{z}}_i^{(1)}) - T^{(1)}(\mathbf{z}_i^{(1)})\|_2.\notag
    \end{equation}
    One may conclude immediately from Proposition \ref{proposition: GMLS error} that $\|\hat T^{(1)}({\mathbf{z}}_i^{(1)}) - T^{(1)}(\mathbf{z}_i^{(1)})\|_2 = \mathcal{O}(\epsilon_{d})$.
    Since the GMLS map $\hat{T}^{(1)}$ is a polynomial function of degree $\ell-1$ (for $\ell\geq 2$) on a neighborhood of $\mathbf{z}_i^{(1)}$, there exists a constant $c>0$ independent of $N$ such that $\|D\hat{T}^{(1)}(y)\|_2\leq c$ for all $y$ in a neighborhood of $\mathbf{z}_i^{(1)}$.
    Therefore, we may write
    \begin{equation}
        \hat T^{(1)}(\tilde{\mathbf{z}}_i^{(1)}) = \hat T^{(1)}(\mathbf{z}_i^{(1)}) + D\hat{T}^{(1)}(\mathbf{z}_i^{(1)})(\tilde{\mathbf{z}}_i^{(1)} - \mathbf{z}_i^{(1)}) + \mathcal{O}(\|\tilde{\mathbf{z}}_i^{(1)} - \mathbf{z}_i^{(1)}\|_2^2).\notag
    \end{equation}
    It follows immediately that $\|\hat T^{(1)}(\tilde{\mathbf{z}}_i^{(1)}) - \hat T^{(1)}(\mathbf{z}_i^{(1)})\|_2 = \mathcal{O}(d^{3/2}\epsilon_{d})$, and so we have 
    \begin{equation}
        \|\hat T^{(1)}(\tilde{\mathbf{z}}_i^{(1)}) - T^{(1)}(\mathbf{z}_i^{(1)})\|_2 = \mathcal{O}(d^{3/2}\epsilon_{d}),\notag
    \end{equation}
    w.p.h.~than $1-\frac{2}{N}$ (because two GMLS applications were required - one for each prolongation step).
    That is, the error introduced by approximating tangent vectors at approximated points in jet space is effectively negligible.
    The previous arguments allow us to conclude that w.p.h.~than $1-\frac{2}{N}$,
    \begin{equation}
        \|\tilde{\mathbf{z}}_i^{(2)} -  \mathbf{z}_i^{(2)} \|_2 = \mathcal{O}(d^3\epsilon_{d}),\notag
    \end{equation}
    because $\mathcal{O}(d^{3/2}\epsilon_{d})$ error is accrued when passing from tangent vectors to derivatives (via error analysis on the resulting linear system).

    The arguments for the error associated with $\tilde{\mathbf{z}}_i^{(1)}$ and $\tilde{\mathbf{z}}_i^{(2)}$ may be applied recursively $p<\infty$ times to show that at the $p$th prolongation step,
$\| \tilde{\mathbf{z}}_i^{(p)} - {\mathbf{z}}_i^{(p)}\|_2 = \mathcal{O}(d^{3p/2}\epsilon_{d})$
    w.p.h.~than $1-\frac{p}{N}$, as was to be shown. 
\end{proof}

We now proceed to derive a bound on $\|\bm\Upsilon\|_2$.
This error bound will allow us to employ Theorem \ref{theorem: sine theorem} to make a final conclusion.

\begin{remark}
    The regularity assumption \ref{assumption} includes a uniform spectral bound on the singular values of $\mathscr{A}_i$, which is a strong assumption.
    Empirically, we find that the condition number of $\tilde{\mathscr{A}}_i$ is $\mathcal{O}(1)$ with no $d$-dependence.
    Therefore, we expect that the $\mathcal{O}(d^{3/2})$ error bound accrued when passing from tangent vector to derivative is overly pessimistic in the majority of cases.
\end{remark}

\begin{lemma}\label{proposition: P error}
    Let the assumptions in Lemma~\ref{proposition: jet data accuracy} hold and that $\ell > d$ for $d \geq 2$.
    Assume that $\|\mathbf{P}\|_2=\mathcal{O}(1)$, and that $L^{(p)}\Psi$ (defined according to \eqref{eq: Psi} and \eqref{eq: matrix P}) is $\mathcal{C}^{2}(M^ {(p)})$.
    Then w.p.h.~than $1-\frac{p+1}{N}$, 
    \begin{equation}
        \tilde{\mathbf{A}} = \mathbf{A} + \bm\Upsilon, \quad \|\bm\Upsilon\|_2 = \mathcal{O}\left(N\left(\frac{\log N}{N} \right)^{\frac{\ell}{d}}\right),\notag
    \end{equation}
    where $\mathbf{A} = \mathbf{P} \mathbf{P}^\top$ and $\tilde{\mathbf{A}} =\tilde{\mathbf{P}} \tilde{\mathbf{P}}^\top$.
    The constant in the big-O notation can depend on $d$, but is independent of $N$.
\end{lemma}

\begin{proof}
Our goal is to provide error estimates for the entries of the objects $\tilde{\mathbf{S}}_i$ and $\tilde{\mathbf{L}}_i$ used to construct $\tilde{\mathbf{P}}$ (defined according to \eqref{eq: stacked matrix}). Recall that we may view $\nabla^{(p)}[\Delta_\nu](\mathbf{z}_i^{(p)})$ as the normal vector to the manifold $\mathscr{S}^{(p)}$ at the point $\mathbf{z}_i^{(p)}$.
We write
\begin{equation}
    \mathbf{S}_i = \begin{bmatrix}
        \nabla^{(p)}[\Delta_1](\mathbf{z}_i^{(p)}) & \dots & \nabla^{(p)}[\Delta_{D_p-d}](\mathbf{z}_i^{(p)})
    \end{bmatrix}\notag
\end{equation}
as the true object to be approximated by $\tilde{\mathbf{S}}_i$.
Numerically, one recovers these normal vectors (or equivalently the corresponding tangent vectors) by applying the GMLS procedure to the available approximated data $\tilde{\mathbf{z}}_i^{(p)}$.
The arguments in Proposition \ref{proposition: jet data accuracy} allow us to conclude that w.p.h.~than $1-\frac{p+1}{N}$, $\|\tilde{\mathbf{S}}_i - \mathbf{S}_i\|_2 = \mathcal{O}(d^{3p/2}\epsilon_{d})$.

Next, by the smoothness of $\Psi$, we may write
\begin{equation}
    L^{(p)}\Psi(\tilde{ \mathbf{z}}_i^{(p)}) = L^{(p)}\Psi({ \mathbf{z}}_i^{(p)}) + DL^{(p)}\Psi({ \mathbf{z}}_i^{(p)})(\tilde{\mathbf{z}}_i^{(p)} - \mathbf{z}_i^{(p)}) + \mathcal{O}(\|(\tilde{\mathbf{z}}_i^{(p)} - \mathbf{z}_i^{(p)})\|_2^2),\notag
\end{equation}
from whence it follows that  $\|\tilde{\mathbf{L}}_i - L^{(p)}\Psi({ \mathbf{z}}_i^{(p)})\|_2 = \mathcal{O}(d^{3p/2}\epsilon_d),$ again by the results of Proposition \ref{proposition: jet data accuracy}.
Notice that if the error bound for $\tilde{\mathbf{S}}_i$ holds, so does the error bound for $\tilde{\mathbf{L}}_i$ because this object depends only on prolongated data of a lower level.

The entries of $\tilde{\mathbf{P}}$ are linear combinations of the entries of $\tilde{\mathbf{L}}_i$ and $\tilde{\mathbf{S}}_i$.
Hence, $\mathcal{O}(d^{3p/2}\epsilon_{d})$ error dominates the entries of $\tilde{\mathbf{P}}$ and we may write w.p.h.~than $1-\frac{p+1}{N}$,
\begin{equation}
    \tilde{\mathbf{P}} = \mathbf{P} + \mathbf{E}, \quad \|\mathbf{E}\|_{\max} = \mathcal{O}(d^{3p/2}\epsilon_d).\notag
\end{equation}
Notice that because the dimension of $\mathbf{E}$ depends on $N$, we have
\begin{equation}
    \|\mathbf{E}\|_2 \leq \sqrt{K}\sqrt{N(D_p-d)}\|\mathbf{E}\|_{\max} = \mathcal{O}(\sqrt{N}\epsilon_d),\notag
\end{equation}
where the constant in the big-O notation could depend on $d$.
A similar argument shows that $\|\mathbf{P}\|_2 = \mathcal{O}(\sqrt{N})$.
Writing $\tilde{\mathbf{P}} = \mathbf{P} + \mathbf{E}$ implies that $\bm\Upsilon = \mathbf{P}^\top \mathbf{E} + \mathbf{E}^\top \mathbf{P} + \mathbf{E}^\top \mathbf{E}$, so
\begin{equation}
    \|\bm\Upsilon\|_2 \leq 2\|\mathbf{P}\|_2 \|\mathbf{E}\|_2 + \|\mathbf{E}\|_2^2 = \mathcal{O}\left(N\epsilon_{d} + N\epsilon_{d}^2\right) = \mathcal{O}(N\epsilon_d),\label{Upsilon}
\end{equation}
where the constant in the big-O notation depends on $d$, but is independent of $N$.
\end{proof}

The results of the previous subsection position us to apply the Davis and Kahan sine theorem \ref{theorem: sine theorem} to the matrices $\mathbf{A}$ and $\tilde{\mathbf{A}}$ to study the fidelity with which our GMLS algorithm recovers the nullspace of $\mathbf{P}$.
We now state and prove our main result.

\begin{theorem}\label{theorem: final error analysis}
Let the assumptions of Lemma \ref{proposition: P error} hold.
Let $\bm\Theta$ denote the diagonal matrix of principal angles between the nullspace of $\mathbf{P}$ and the numerical nullspace of $\tilde{\mathbf{P}}$ (or equivalently $\mathbf{A}=\mathbf{P} \mathbf{P}^\top$ and $\tilde{\mathbf{A}}=\tilde{\mathbf{P}}\tilde{\mathbf{P}}^\top$, respectively).
Denote the eigenvalues of $\mathbf{A}$ as $0=\lambda_1=\dots=\lambda_r<\lambda_{r+1}\leq \dots \leq \lambda_K$, and assume that $\lambda_{r+1} = \omega(\|\bm \Upsilon\|_2)$, as $\|\bm \Upsilon\|_2 \to 0$, where $\|\bm\Upsilon\|_2$ is defined in \eqref{Upsilon}.
Then w.p.h.~than $1-\frac{p+1}{N}$,
\begin{equation}
    \|\sin\bm\Theta\|_2 = \mathcal{O}\left(N\left(\frac{\log N}{N}\right)^{\ell/d} \right)\notag
\end{equation}
as $N\to \infty$.
The big-$\mathcal{O}$ notation could depend on $d$ and inversely on the spectral gap $\lambda_{r+1}$, but is independent of $N$.
\end{theorem}

Before we proof this theorem, we note that the spectral gap condition, $\lambda_{r+1}= \omega (\|\bm\Upsilon\|_2)$ is a reasonable assumption since $\|\bm\Upsilon\|_2 \to 0$ as $N\to 0$ (see \eqref{Upsilon}), so one can choose $N$ large enough such that this condition is satisfied in practice.  

\begin{proof}
    Note that both $\mathbf{A}$ and $\tilde{\mathbf{A}}$ are Hermitian matrices that satisfy the decomposition in \eqref{eq:def A}.
    \comment{, and hence we may write
     \begin{equation}
        \begin{split}
            \mathbf{A} &= \mathbf{H}_0 \mathbf{A}_0 \mathbf{H}_0^* + \mathbf{H}_1 \mathbf{A}_1 \mathbf{H}_1^*,
            \\
            \tilde {\mathbf{A}} &= \mathbf{F}_0 \bm{\Lambda}_0 \mathbf{F}_0^* + \mathbf{F}_1\bm{\Lambda}_1 \mathbf{F}_1^*,
        \end{split}\notag
    \end{equation}
    as a unitary diagonalization of $\mathbf{A}$ and $\tilde{\mathbf{A}}$.}
    Particularly, $\mathbf{A}_1$ is a diagonal matrix whose entries are the nonzero eigenvalues of $\mathbf{A}$, and the columns of $\mathbf{H}_1$ are the corresponding eigenvectors.
    The matrix $\mathbf{A}_0$ is the zero matrix (collecting the zero eigenvalues of $\mathbf{A}$), and the columns of $\mathbf{H}_0$ are the corresponding eigenvectors.
    The notation for $\tilde{\mathbf{A}}$ represents the perturbed counterparts to these quantities.
    Let $\tilde{\lambda}^*$ be the smallest entry of $\bm{\Lambda}_1$ and define 
      $  \mathbf{R} = (\mathbf{A}+\bm\Upsilon)\mathbf{H}_0 - \mathbf{H}_0 \mathbf{A}_0
    $ as in Theorem \ref{theorem: sine theorem}. 
    In relation to Theorem~ \ref{theorem: sine theorem}, we have $a,b=0$. 
    Choosing $\delta = \tilde{\lambda}^*>0$, it is clear that spectrum of $\bm{\Lambda}_1$ lies outside of $(-\tilde{\lambda}^*,\tilde{\lambda}^*)$.

    Using the fact that $\mathbf{H}_1^\top \mathbf{H}_0 = 0$, it follows that $\mathbf{A}\mathbf{H}_0 = 0$.
    Therefore,
    \begin{equation}
        \begin{split}
            \| \mathbf{R}\|_2 &= \|\bm\Upsilon \mathbf{H}_0\|_2 \leq \|\bm\Upsilon \|_2\|\mathbf{H}_0\|_2 \leq \|\bm\Upsilon \|_2. 
        \end{split}\notag
    \end{equation}
    Our next step is to attain a lower bound for $\tilde{\lambda}^*$. Note that $\lambda_{r+1}\leq \lambda_{r+2}\leq \cdots \leq \lambda_K$.
    Upon perturbation, each of these eigenvalues can increase or decrease by at most $\|\bm\Upsilon\|_2$ (see \cite{stewart1998perturbation,weyl1912asymptotische} for further details).
    Because $\lambda_{r+1}$ is the smallest eigenvalue of $\mathbf{A}_1$, $\lambda_{r+1}-\|\bm\Upsilon\|_2$ is a lower bound for all eigenvalues of $\bm\Lambda_1$, and so 
    \[
    \tilde{\lambda}^* \geq \lambda_{r+1}-\|\bm\Upsilon\|_2.
    \]
By Theorem~\ref{theorem: sine theorem} and Lemma~\ref{proposition: P error}, w.p.h.~than $1-\frac{p+1}{N}$,
    \begin{equation}
        \begin{split}
            \|\sin\bm\Theta\|_2 &\leq \frac{\|\bm\Upsilon\|_2}{\tilde{\lambda}^*}  \leq \frac{\|\bm\Upsilon\|_2}{\lambda_{r+1}-\|\bm\Upsilon\|_2} 
            = \frac{\|\bm\Upsilon\|_2}{\lambda_{r+1}} \left(1 +             \mathcal{O}\left(\frac{\|\bm\Upsilon\|_2}{\lambda_{r+1}}\right) \right)
        =    \mathcal{O}\left(\frac{N\epsilon_{d}}{\lambda_{r+1}}\right),
        \end{split} \notag
    \end{equation}   
    where we have used the assumption that $\lambda_{r+1} = \omega(\|\bm\Upsilon\|_2)$ as $\|\bm\Upsilon\|_2 \to 0$ to allow the Taylor expansion $\frac{1}{1-\varepsilon} = 1+\mathcal{O}(\varepsilon)$, with $\varepsilon = \frac{\|\bm\Upsilon\|_2}{\lambda_{r+1}} \ll 1$. 
    Note that the upper bound is inversely proportional to the spectral gap $\lambda_{r+1}$, which is independent of $N$. 
\end{proof}

Theorem \ref{theorem: final error analysis} establishes a bound on the error in our numerical procedure under the assumption that the symmetries of the system, i.e., the vector field $X_g$, can be represented exactly in a finite-dimensional basis. 
Otherwise, one may instead choose the best approximate vector field $X_g^*$ that minimizes the residual of the infinitesimal invariance condition,
\begin{equation}\label{eq: bias error}
    X_g^* = \arg\min_{c \in \mathbb{R}^K}
        \|L^{(p)}(\Psi c) \cdot \nabla^{(p)}[\Delta_\nu]\|_2,
\end{equation}
where the columns of $\Psi$ (defined in \eqref{eq: Psi}) span the approximation space such that $L^{(p)}\Psi \in C^2(M^{(p)})$. 
We remark that \eqref{eq: bias error} is a standard problem in approximation theory \cite{davis1970rotation} and will not be pursued further in this work.

The sampling assumption in this theorem can be relaxed to non-uniform distributions with bounded density functions. 
In such a case, the same error $\epsilon_d$ is valid in Proposition~\ref{proposition: GMLS error}; see Lemma D.3 in \cite{huang2025learning}. 
While one can relax the i.i.d. assumption, to the best of our knowledge additional assumptions, such as mixing conditions, are needed to access the non-i.i.d.~Bernstein type concentration inequality needed for such an analysis \cite{hang2017bernstein}. 
We do not pursue this avenue here because mixing assumptions are often too strict for our applications. 
For non-i.i.d.~data, one can also consider the deterministic bound in \cite{mirzaei2012generalized} and replace the $N^{-1}\log(N)$ factor in Proposition~\ref{proposition: GMLS error} with the fill distance. The resulting upper bound is a function of both $N$ and the fill distance, where the latter dependence on $N$ for arbitrary manifolds is unspecified. 

%% file: arxiv_sec5_results.tex
In this section, we verify the efficacy of our numerical procedure on several ordinary and partial differential equations.
Unless otherwise stated, we take $\epsilon=10^{-5}$ to be error threshold for our numerical algorithm.
All codes used to generate the figures in this section are publicly available at \url{https://github.com/MaxKreider/SymmetryAndScaling}.

\subsection{First order linear ODE}\label{ex: particular vs family}

As a pedagogical first example, we consider the first order, linear system,
 $\Delta(x,u,u_x) = u - u_x = 0$,
with general solution $u(x,C) = Ce^x$, where $C\in \mathbb{R}$ is an integration constant.
We adopt a linear ansatz for the form of the infinitesimal generators
\begin{equation}\label{eq: ex simple ansatz}
\begin{split}
    X_g = \xi_1(x,u)\partial_x + \eta_1(x,u)\partial_u \equiv (c_0+c_1 x+c_2u)\partial_x + (c_3+c_4x+c_5 u)\partial_u.
\end{split}
\end{equation}
This ansatz leads to two independent infinitesimal generators 
\begin{equation}\label{eq: ex simple gen}
    \begin{split}
        X_g^{(1)} &= c_0 \partial_x, \quad X_g^{(2)} = (c_5 u) \partial_u,
    \end{split}
\end{equation}
which indicates that the ODE admits both translation and scaling symmetry (see Appendix \ref{Appendix: example} for a derivation).

To apply our algorithm, we consider a family of solution curves \(\{(x_{i}, C_j, u_{i,j})\}\) defined by $u_{i,j} = u(x_i,C_j)= C_j e^{x_i},$ where the parameters $C_j$ are randomly drawn $1001$ times from the uniform distribution on $[1,2]$, and the points $x_i$ are randomly drawn $501$ times from the uniform distribution on $[-1,1]$.
This discretization generates a $501\times 1001$ non-uniform grid for the effective independent variables.
Our algorithm, with $\mathfrak K=15$ and $\ell=4$, generates approximations for the jet space data $\{(x_i,C_j, u_{i,j}, (\tilde u_x)_{i,j})\}$.
To generate the normal vectors at the last step, we modify $\mathfrak K = 25$ and $\ell=3$.
We emphasize that the free integration constant is treated as an effective independent variable, and therefore the prolongated data forms a two-dimensional manifold embedded in $\mathbb{R}^4$.

We approximate the coefficients $c_i$ in \eqref{eq: ex simple ansatz} by finding the numerical nullspace of $\tilde{\mathbf{P}}$, which is constructed according to \eqref{eq: stacked matrix}.
The singular values and singular vectors of the matrix $\tilde{\mathbf{P}}$ are shown in Fig.~\ref{fig: ex simple}(a)-(c).
We observe a clear spectral gap showing two nearly trivial singular values in \ref{fig: ex simple}(a).
The corresponding right singular vectors (Fig.~\ref{fig: ex simple}(b)-(c)) give a basis for the numerical nullspace of $\tilde{\mathbf{P}}$, and provide approximations for the coefficients $c_i$ in \eqref{eq: ex simple ansatz}.
Indeed, we recover \eqref{eq: ex simple gen}, indicating that our algorithm successfully identified the symmetry groups.

It is instructive to repeat this experiment, but with given data $\{(x_i,u_i)\}$ of the form $ u_i = u(t_i,1) = e^{x_i},$ where the $x_i$ are randomly sampled 1000 times from the uniform distribution on $[-2,1]$, but $C=1$ is held fixed.
We prolongate this data using our algorithm with $\mathfrak K = 10$ and $\ell = 3$, and note that now the jet space data $\{(x_i, u_i, (\tilde u_x)_i)\}$ forms a one-dimensional curve embedded in $\mathbb{R}^3$.
The spectral properties of the resulting $\tilde{\mathbf{P}}$ are summarized in Fig.~\ref{fig: ex simple}(d)-(f).
In this case, there is only one nearly trivial singular value separated from the rest, and the corresponding singular vector gives \textit{a fixed relation $c_0=c_5=c^*$} between the ansatz coefficients.
To understand this behavior, we analyze the single infinitesimal generator
$X_g = (c^*)\partial_x + (c^*u)\partial_u.$
This generator corresponds to a system of ODEs
        $\frac{d\tilde x}{ds} = c^*, \frac{d\tilde u}{ds} = c^*u,$
whose solution (with initial conditions $\tilde x(0)=x$ and $\tilde u(0) =u$) is $\tilde x(s) = c^*s + x$, $\tilde u(s) = u e^{c^*s}.$
The solution curve $u(x)=e^x$ under \textit{simultaneous perturbation} $(x,u) \to (\tilde x, \tilde u) = (c^*s + x, ue^{c^*s})$ remains invariant because
\begin{equation}
    \begin{split}
        \tilde u (\tilde x) = u(x) e^{c^*s}
        = e^x e^{c^*s}
        = e^{\tilde x - c^*s} e^{c^*s}
        = e^{\tilde x}.
    \end{split}\notag
\end{equation}
Therefore, simultaneous perturbation in $x$ and $u$ is necessary for the particular solution $u(x)=e^x$ to remain invariant.
In contrast, it is straightforward to verify that the family of solutions $u(x)=Ce^x$ for $C\in \mathbb{R}$ remains invariant under independent perturbations in either $x$ or $u$.
That is, the symmetry properties of a family of solutions may differ from that of a particular member of this family.
Our algorithm detects symmetries reflective of the data, not of the underlying dynamical system.

For completeness, we perform a convergence study for the case of fixed $C = 1$. 
The error metric is $\|\sin(\bm\Theta)\|_2$, which quantifies the angle between the numerically computed and true nullspaces of $\tilde{\mathbf P}$ and $\mathbf P$ (see Theorem~\ref{theorem: final error analysis}). 
In this setting, the true nullspace of $\mathbf P$ is spanned by $[1,0,0,0,0,1]^\top$. 
We randomly sample $x_i$ from the uniform distribution on $[-2,1]$ with $N \in \{2^q\times 10 \}_{q=3,\ldots, 12}$, and for each $N$ report the mean error over 100 independent trials. 
The parameters $\mathfrak K = 10$ and $\ell = 3$ are fixed throughout. 
The results, shown in Fig.~\ref{fig: ex simple}(f), indicate that the theoretical error bound in Theorem~\ref{theorem: final error analysis} is conservative for this example.

\begin{figure}[ht]
{\scriptsize \centering
\begin{tabular}{ccc}
\normalsize (a)  & \normalsize (b) & \normalsize (c) \\
    \includegraphics[scale=.28]{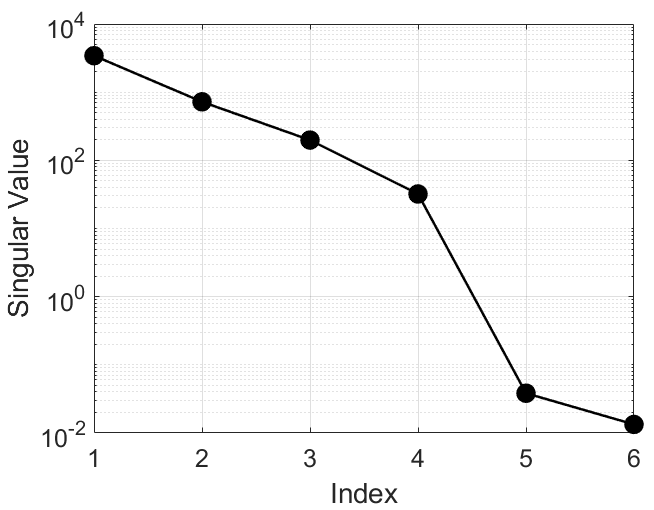} &
    \includegraphics[scale=.28]{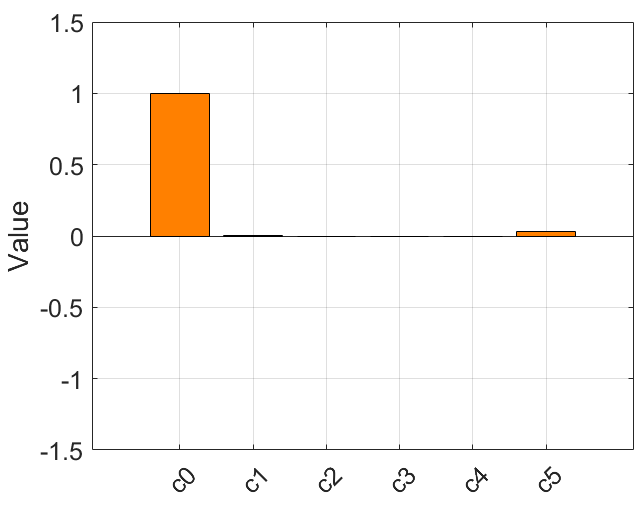} &
    \includegraphics[scale=.28]{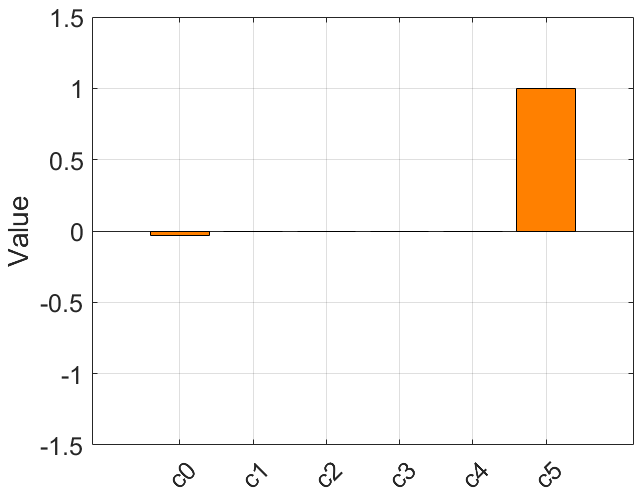} \\
    \normalsize (d)  & \normalsize (e) & \normalsize (f) \\
    \includegraphics[scale=.28]{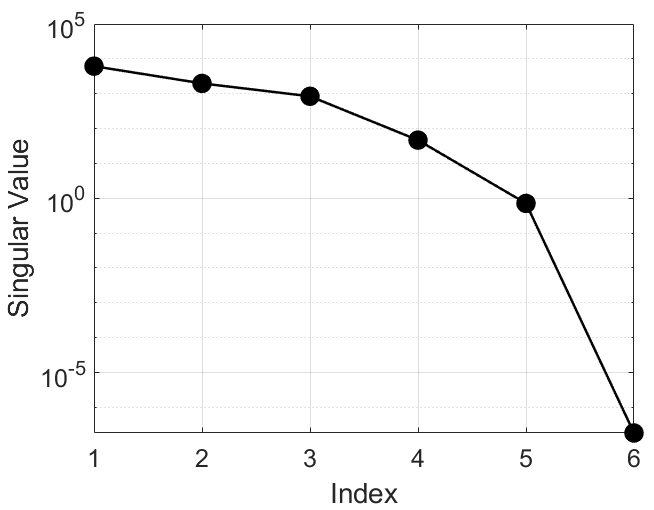} &
    \includegraphics[scale=.28]{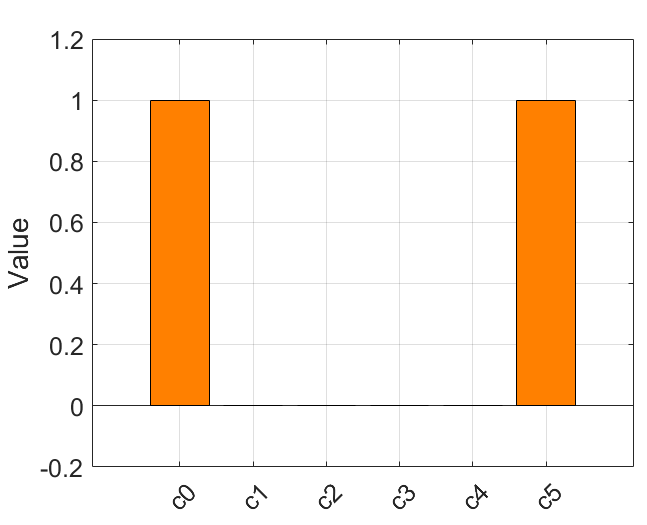} &
    \includegraphics[scale=.28] {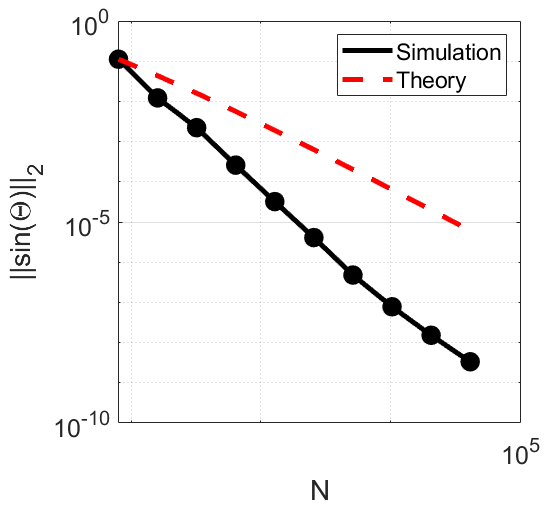} 
    \end{tabular}} 
    \caption{Singular value decomposition of $\tilde{\mathbf P}$ for the linear ODE, shown for a family of solutions \textbf{(a-c)} and with $C=1$ fixed \textbf{(d-f)}.
    \textbf{(a)} Semilog plot of singular values $\sigma_i$, showing two nearly vanishing modes. 
    \textbf{(b)} and \textbf{(c)} Bar plots of the right singular vectors corresponding to $\sigma_5$ and $\sigma_6$, which span the numerical nullspace of $\tilde{\mathbf P}$.
    \textbf{(d)} Semilog plot of $\sigma_i$ for fixed $C$, revealing a single nearly trivial mode. 
    \textbf{(e)} Bar plot of the corresponding right singular vector. 
    \textbf{(f)} Comparison of numerical (black) and theoretical (red, rescaled) error $\|\sin(\bm\Theta)\|_2$, averaged over 100 trials for each $N$.}
    \label{fig: ex simple}
\end{figure}

\subsection{Stuart-Landau oscillator}

Recall the Stuart-Landau oscillator \eqref{eq: SL oscillator} in Example~\ref{exp_de}, with general solution,
\begin{equation}
    \begin{split}
        x(t,C_1,C_2) = \frac{\cos(-t+C_1)}{\sqrt{1-\left(1-\frac{1}{C_2^2}\right)e^{-2t}}}, \quad y(t,C_1,C_2) = \frac{\sin(-t+C_1)}{\sqrt{1-\left(1-\frac{1}{C_2^2}\right)e^{-2t}}}.
    \end{split} \label{SL sol}
\end{equation}
Note that this system admits a stable limit-cycle solution which is a circle of unit radius.
We adopt a linear ansatz of the form
\begin{equation}
    X_g 
    =(c_0+c_1 t+c_2x + c_3y)\partial_t + (c_4+c_5t+c_6 x + c_7 y)\partial_x + (c_8+c_9t+c_{10} x + c_{11} y)\partial_y.
\end{equation}
This ansatz leads to two independent infinitesimal generators
\begin{equation}\label{eq: SL ansatz thing}
    \begin{split}
        X_{g,1}(t,x,y) &= c_0 \partial_t, \quad X_{g,2}(t,x,y) = (-c_7 y) \partial_x + (c_{10} x) \partial_y,
    \end{split}
\end{equation}
which indicates that \eqref{eq: SL oscillator} admits both $SO(2)$ and time-translation symmetry (recall Example \ref{ex: SL}). To apply our algorithm, we consider a family of solution curves $\{(t_i,(C_1)_j,(C_2)_k,x_{i,j,k},y_{i,j,k})\}$ defined by $x_{i,j,k} = x(t_i,(C_1)_j,(C_2)_k)$, and $y_{i,j,k} = y(t_i,(C_1)_j,(C_2)_k)$ with \eqref{SL sol}.
We randomly sample $x_i$ from the uniform distribution on $[0,\pi]$ ($N_x$ points), $(C_1)_j$ from the uniform distribution on $[0,2\pi]$ ($N_1$ points), and $(C_2)_k$ from the uniform distribution on $[1,1.3]$ ($N_2$ points), forming a non-uniform $N_x \times N_1 \times N_2$ mesh for the effective independent variables.
With $\mathfrak K = 40$ and $\ell = 4$ fixed, we perform a convergence study over $20$ independent trials for $N_x \in \{225,300,600\}$, $N_1 \in \{1,2,3\}$, and $N_2 \in \{95,126,252\}$.
The true nullspace of $\mathbf P$ is spanned by the infinitesimal generators in~\eqref{eq: SL ansatz thing}, see results in Fig.~\ref{fig:convg}(a).

We also study convergence for $N$ points randomly sampled from the uniform distribution on the unit circle, corresponding to the particular solution with $C_1 = 0$ and $C_2 = 1$ fixed.
We take $N \in  \{2^q\times 10 \}_{q=3,\ldots, 12}$, with $\mathfrak K = 10$ and $\ell = 3$ fixed.
In this case, the true nullspace is spanned by the vector $c$, where $c_1 = -c_7 = c_{10}$ and all other entries zero.
Results are shown in Fig.~\ref{fig:convg}(b).

\subsection{Transport equation}

We consider the transport equation $u_t - u_x = 0,$ and adopt a linear ansatz for the infinitesimal generator
\begin{equation}
        X_g 
        =(c_0 + c_1t + c_2 x + c_3 u)\partial_t + (c_4 + c_5t + c_6x + c_7u)\partial_x + (c_8 + c_9t + c_{10}x + c_{11}u)\partial_u.
\end{equation}
Under this ansatz, the prolongated generator takes the form,
\BEA
        X_g^{(1)} &=& X + \eta_{1,1}(t,x,u,u_t,u_x)\partial_{u_t} + \eta_{1,2}(t,x,u,u_t,u_x)\partial_{u_x},
    \notag\\
        \eta_{1,1} &=& 
        c_9 + c_{11}u_t - c_1 u_t - c_3 (u_t)^2 - c_5 u_x - c_7 u_xu_t,\notag
        \\
        \eta_{1,2} &=& 
        c_{10} + c_{11}u_x - c_6 u_x - c_2 u_t - c_7 (u_x)^2 - c_3 u_xu_t.
        \notag
\EEA
Analytically, one can show (after tedious calculation) that the infinitesimal invariance condition \eqref{eqn_sym} leads to $c_9 = c_{10}$, $c_2 + c_6 = c_1 + c_5,$
with all other coefficients free.
The transport equation admits many different particular solutions, each of which can exhibit a variety of symmetries.
For concreteness, we will explore the symmetries of $u(t,x) = \sin(t+x)$.
We randomly sample $t_i$ and $x_j$ from the uniform distribution on $[-1/2,1/2]\times [-1/2,1/2]$ to form a non-uniform $N\times N$ grid, and define $u_{i,j} = \sin(t_i + x_j)$, so that sampled data takes the form $\{(t_i,x_j,u_{i,j})\}$.
We perform a convergence analysis with $N\in \{56,80,113,160\}$, $\mathfrak K=20$ and $\ell=3$ fixed, and average results over 100 trials.
Here, the true nullspace is spanned by one vector with nonzero entries $c_1=c_6 = -c_1 = -c_5$, and all other entries zero.
The results are presented in Fig.~\ref{fig:convg}(c).

\subsection{Heat equation}\label{subsection: ex heat equation}

As our final example, we consider the heat equation $u_t - u_{xx} = 0$.
We adopt a linear ansatz for the infinitesimal generator
\BEA        X_g = 
         (c_0 + c_1t + c_2 x + c_3 u)\partial_t + (c_4 + c_5t + c_6x + c_7u)\partial_x + (c_8 + c_9t + c_{10}x + c_{11}u)\partial_u.\notag\EEA

Analytically, one can show (after tedious calculation) that the invariance condition $X_g^{(2)}[\Delta]=0$ whenever $\Delta = 0$ gives
\begin{equation}
    c_2=c_3=c_5=c_7=c_9=0, \quad c_1=2c_6, \notag
\end{equation}
with all other coefficients free (see \cite{olver1993applications} for further details).
The heat equation admits many different particular solutions, each of which can exhibit a variety of symmetries.
For concreteness, we will explore the symmetries of 
$u(t,x) = \frac{1}{\sqrt{4\pi t}}\exp(-x^2/(4t)).$
We sample $t_i$ and $x_j$ randomly from the uniform distribution on $[1,2]\times [1,2]$ to form a non-uniform $N\times N$ grid, and define $u_{i,j} = u(t_i,x_j)$,
so that the sampled data takes the form $\{(t_i,x_j,u_{i,j})\}$.
We perform a convergence analysis with $N\in\{56,80,113,160\}$, $\mathfrak K=40$ and $\ell=4$ fixed, and average results over 100 trials.
In this case, the true nullspace is spanned by one vector with nonzero entries $c_1=2c_6 = -c_{11}$, and all other entries zero.
The results are shown in the last panel of Fig.~\ref{fig:convg}.

\begin{figure}[ht]
{\scriptsize \centering
\begin{tabular}{cccc}
\normalsize (a)  & \normalsize (b) & \normalsize (c) & \normalsize (d) \\
    \includegraphics[scale=.28]{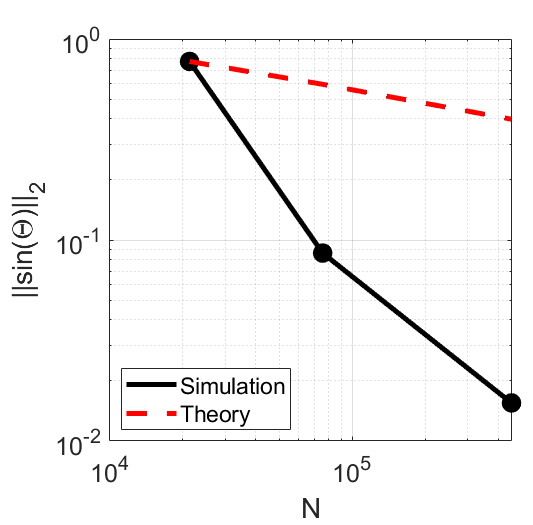} &
    \includegraphics[scale=.28]{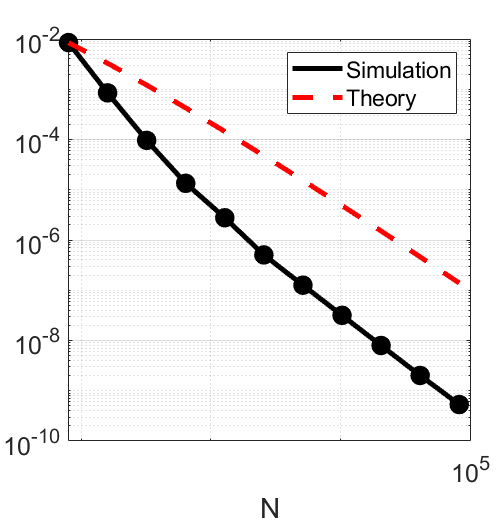} &
    \includegraphics[scale=.28]{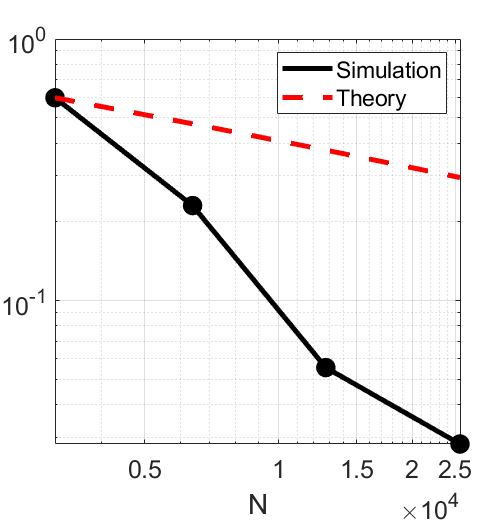} &
    \includegraphics[scale=.28]{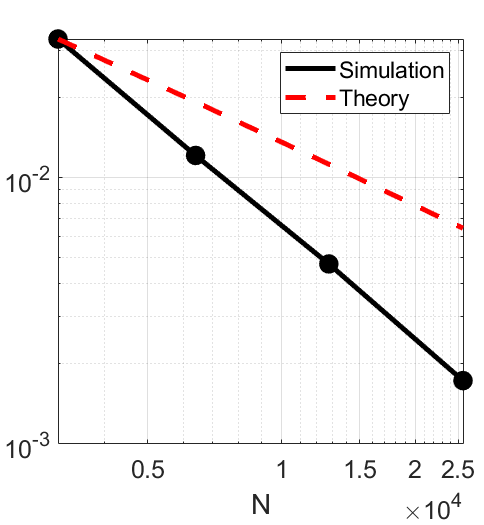}
    \end{tabular}}
    \caption{Convergence results for: \textbf{(a)} the SL oscillator with free integration constants, \textbf{(b)} the SL oscillator with fixed integration constants, \textbf{(c)} the transport equation, and \textbf{(d)} the heat equation.
    In each case, the theoretical error bound (red) is rescaled to match the initial value of the numerical error (black).
    The observed convergence rates consistently exceed the theoretical predictions, indicating that the bound is conservative.}
    \label{fig:convg}
\end{figure}

%% file: arxiv_sec6_discussion.tex
Symmetry plays a central role in the study of DEs, providing a means to identify conserved quantities and reduce model complexity.
Classical Lie group theory provides an analytical means to identify system symmetries via the infinitesimal invariance criteria, provided that underlying system dynamics are known.
However, in many modern applications the governing equations are unknown, and only scattered data is available.
In this setting, symmetries can reveal structure in complicated high-dimensional datasets.
Despite its importance, a rigorous and practical framework to study symmetries directly from data remains underdeveloped.

In this work, we bridge this gap by developing a numerical framework for identifying symmetries directly from scattered data sampled on an unknown manifold. 
Our method employs a manifold learning technique, Generalized Moving Least Squares (GMLS), to numerically prolongate the scattered data, after which the infinitesimal invariance condition is discretized into a linear system encoding the underlying Lie point symmetries.
Conventionally, analyzing the invariance condition analytically requires knowledge of the underlying vector field and its gradient.
One key contribution in this work is to interpret these gradients as the normal vector of the solution manifold of the DEs, which can be directly estimated from the prolongated data by GMLS without knowledge of the underlying dynamics.
Building on GMLS, the second key contribution of this work is to establish and numerically verify theoretical convergence rates for our algorithm, which are largely absent in the literature \cite{hu2025explicit}.
We demonstrated the effectiveness of this approach on both ordinary and partial differential equations, showing that it can distinguish between symmetries associated with particular solutions and those shared across families of solutions.

Our framework is built around the equation-free estimation of the gradients of the DEs.
These gradients also play a key role in the analysis of more general types of symmetries, such as contact and generalized (Lie–B\"{a}cklund) symmetries where the invariant transformation involves higher-order derivatives, and conditional symmetries where the DE solution is only partially observed \cite{cantwell2002introduction,olver1993applications}.
A natural next step is to extend the present framework to the case of these more general symmetries.

Symmetries and scaling laws are often prevalent in real-world applications where the underlying dynamics are either unknown or too complicated to work with directly, such as fluid mechanics and other multi-disciplinary problems  \cite{huang2020multi}.
However, such data is noisy, due to numerical errors in DE solutions or measurement errors, and often highly nonlinear.
Extending the present framework to (i) ensure robustness in the presence of noise and (ii) capture nonlinear representations of infinitesimal generators would provide a promising approach for discovering symmetries in real-world systems. 

%% file: arxiv_sec7_appendices.tex
\section{Lie group theory}\label{Appendix: Lie group theory definitions}

In this section, we recall several useful definitions in Lie group theory for the benefit of the unfamiliar reader.

\begin{definition}
An \textbf{r-parameter local Lie group} is a smooth ($\mathcal{C}^\infty$), $r$-dimensional manifold $G$ equipped with:
\begin{description}
\item[(i)] a distinguished point $e\in G$ as the identity,
\item[(ii)] a smooth local multiplication $m:U\to G$, $m(x,y)=xy$, defined on an open neighborhood $U\subset G\times G$ of $(e,e)$, and
\item[(iii)] a smooth local inversion $i:V\to G$, $i(x)=x^{-1}$, defined on an open neighborhood $V\subset G$ of $e$.
\end{description}
Furthermore, there exists an open neighborhood $W \subset G$ of $e$ such that for all $x,y,z \in W$, the following three axioms hold whenever the expressions are defined: 
\begin{description}
\item[(iv)] \textit{(local associativity)}  $(xy)z = x(yz)$ when $(x,y),(xy,z),(y,z),(x,yz)\in U$,
\item[(v)] \textit{(local identity)} $ex=xe=x$ when $(e,x),(x,e)\in U$, and
\item[(vi)] \textit{(local inverse)} $x^{-1}x=xx^{-1}=e$ when $x\in V$ and $(x^{-1},x),(x,x^{-1})\in U$.
\end{description}
\end{definition}

\begin{definition}
A \textit{local group of transformations} on a smooth manifold $M$ consists of an open set $\{0\}\times M \subset \Omega\subset \bR^r\times M$ (here $0$ denotes the identity element) and a smooth map $\Phi:\Omega\to M, (a,x)\mapsto \Phi(a,x)$, satisfying the following properties:
\begin{description}
\item[(i)] \textit{(identity)} $\Phi(0,x)=x$ for all $x\in M$, and
\item[(ii)] \textit{(local composition)} there exists an open neighborhood $\{0\}\subset A\subset\bR^r$, a smooth local group law $\ast:A\times A\to A$, and a local inverse $(\cdot)^{-1}$, so that $\Phi(a,\Phi(b,x))=\Phi(a\ast b,x)$ and $\Phi(a^{-1},\Phi(a,x))=x,$ whenever all terms are defined.
\end{description}
\end{definition}

\section{Symmetry in algebraic equations}\label{Appendix: symmetry algebraic equations}

Here, we review symmetry in algebraic equations, which provides a foundation for the study of symmetries in DEs.

Let $G$ be a local group of transformations acting on a manifold $M$.
We will use the shorthand notation $g\cdot x\equiv \Phi(g,x)$ for $g\in G$ and $x\in M$.
A subset $\mathscr{S}\subset M$ is called \textit{G-invariant}, and $G$ is called a \textit{symmetry group} of $\mathscr{S}$, if $g\cdot x\in \mathscr{S}$ for all $x\in \mathscr{S}$ and $g\in G$.
Furthermore, a function $F:M\to M'$, where $M'$ is another manifold, is called a \textit{G-invariant function} if $F(g\cdot x) = F(x)$ for all $x\in M$ and all $g\in G$.
The following theorem can be found in \cite{olver1993applications}.

\begin{theorem}\label{Theorem: algebraic symmetry}
    Let $G$ be a connected local Lie group of transformations acting on the $d$-dimensional manifold $M$. Let $F:M\to \mathbb{R}^q$, $q\leq d$, define a system of algebraic equations
    \begin{equation}
        F_\nu(x) = 0, \quad \nu = 1,\dots,q,\notag
    \end{equation}
    and assume that the system is of maximal rank, meaning that the Jacobian matrix $(\partial F_\nu/\partial_{x_k})$ is of rank $q$ at every solution $x$ of the system.
    Then $G$ is a symmetry group of the system if and only if
    \begin{equation}
        X_g[F_\nu](x) = 0,\quad \nu = 1,\dots,q, \quad \text{ whenever } F_\nu(x)=0,
    \end{equation}
    for every infinitesimal generator $X_g$ of $G$.
\end{theorem}
Theorem \ref{Theorem: algebraic symmetry} provides a means to study the symmetry of systems of algebraic equations, as we now illustrate.

\begin{example}
Recall the Lie group $SO(2)$ group, with infinitesimal generator $X_g(x_1,x_2)=-x_2\partial_{x_1} + x_1\partial_{x_2}$.
Consider the algebraic equation
\begin{equation}
    F(x_1,x_2) = x_1^4 + x_1^2x_2^2 + x_2^2 - 1 = 0.\notag
\end{equation}
A straightforward computation shows that $X_g[F]=0$ whenever $F(x_1,x_2)=0$.  Furthermore, the Jacobian matrix 
\begin{equation}
    DF(x_1,x_2) = \begin{bmatrix}
        4x_1^3+2x_1x_2^2 & 2x_1^2x_2+2x_2
    \end{bmatrix},\notag
\end{equation}
is rank 1 except when $x_1=x_2=0$, which does not satisfy $F(x_1,x_2)=0$.
Therefore, the maximal rank condition is satisfied, and we conclude that the zero set of $F(x_1,x_2)$ exhibits $SO(2)$ rotational symmetry.

This symmetry is also evident from the factorization $F(x_1,x_2) = (x_1^2+1)(x_1^2+x_2^2-1) = 0$, which shows that the zero set is the unit circle, and is invariant under rotations.
\end{example}

\section{Illustrative example of Lie symmetry computation}\label{Appendix: example}

In this section, we clarify the notation introduced in \S\ref{subsec: discovering symmetries in DEs} by working through a concrete example illustrating the computation of Lie symmetries for a simple differential equation.

\begin{example}\label{ex: simple}
    Consider the system
    \begin{equation}
        \Delta(x,u,u_x) = u - u_x = 0.\notag
    \end{equation}
    We adopt a linear ansatz for the infinitesimal generator
    \begin{equation}
\begin{split}
    X_g &= \xi_1(x,u)\partial_x + \eta_1(x,u)\partial_u
    \\
    &\equiv (c_0+c_1 x+c_2u)\partial_x + (c_3+c_4x+c_5 u)\partial_u, \notag
\end{split}
\end{equation}
and so we write
\begin{equation}
    \Psi(x,u) = \begin{bmatrix}
        1 & x & u & 0 & 0 & 0
        \\
        0 & 0 & 0 & 1 & x & u
    \end{bmatrix}, \quad c = \begin{bmatrix}
        c_0 & c_1 & c_2 & c_3 & c_4 & c_5
    \end{bmatrix}^\top. \notag
\end{equation}
The form of the prolongated generator is completely determined by the generator in original coordinates.
According to equation (2.3)
, we write
\begin{equation}
    X_g^{(1)} = X_g + \eta_{1,1}(x,u,u_x)\ppf{}{u_x},\notag
\end{equation}
with
\begin{equation}
    \eta_{1,1}(x,u,u_x) = (\eta_1)_x + (\eta_1)_u u_x - (\xi_1)_x u_x - (\xi_1)_u (u_x)^2 = c_4 + c_5 u_x - c_1 u_x - c_2 (u_x)^2.\notag
\end{equation}
Therefore,
\begin{equation}
    L^{(1)}\Psi = \begin{bmatrix}
        1 & x & u & 0 & 0 & 0
        \\
        0 & 0 & 0 & 1 & x & u
        \\
        0 & -u_x & -u_x^2 & 0 & 1 & u_x
    \end{bmatrix}, \quad \nabla^{(1)}[\Delta] = \begin{bmatrix}
        0 \\ 1 \\ -1
    \end{bmatrix}.\notag
\end{equation}
We find that along the constraint $\Delta = 0$,
\begin{equation}
    P = (L^{(1)}\Psi)^\top \nabla^{(1)}[\Delta] = \begin{bmatrix}
        0 & -x & -u+u^2 & 1 & x-1 & 0
    \end{bmatrix}^\top.\notag
\end{equation}
The nullspace of $P$ is 
\begin{equation}
    \text{null }P = \Big\{\begin{bmatrix}
        1 & 0 & 0 & 0 & 0 & 0
    \end{bmatrix}, \begin{bmatrix}
        0 & 0 & 0 & 0 & 0 & 1
    \end{bmatrix}\Big\},\notag
\end{equation}
which recovers two independent infinitesimal generators
\begin{equation}\label{eq: gen}
    \begin{split}
        X_g^{(1)} &= c_0 \partial_x, \quad X_g^{(2)} = (c_5 u) \partial_u.\notag
    \end{split}
\end{equation}
The generators in \eqref{eq: gen} can be interpreted as a system of ODEs with initial condition $\tilde x(0) = x$ and $\tilde u(0) = u$
\begin{equation}
    \begin{split}
        \frac{d\tilde x}{ds} &= c_0, \quad
        \frac{d\tilde u}{ds} = c_5 u,\notag
    \end{split}
\end{equation}
with solutions $\tilde x(s) = c_0s + x$ and $\tilde u(s) = u e^{c_5 s}$.
A straightforward computation verifies that each perturbation $x \to \tilde x(s)$ and $u \to \tilde u(s)$ maps solutions of $\Delta$ to solutions of $\Delta$ (independently of each other).
That is, ``translation in $x$ (time)'' and ''multiplicative scaling'' are symmetries of the zero set of $\Delta$, $u(x)=e^x$.
\end{example}

\section{GMLS algorithm}\label{Appendix: GMLS alg}

Here, we provide Algorithm \ref{alg:gmls}, which describes a GMLS procedure adapted from \cite{zhang2025geometric}.

\begin{algorithm}
\caption{Generalized Moving Least Squares Algorithm}\label{alg:gmls}
\begin{algorithmic}[1]
\Require 
Distinct data points $\{\mathbf y_i\}_{i=1}^N$ on a $d$-dimensional manifold, the stencil $\mathcal{K}_i=\{\mathbf y_{i_j}\}_{j=1}^{\mathfrak K}$ of 
$\mathfrak{K}\ll N$ nearest neighbors to $\mathbf y_i$, 
the SVD bases 
$\{\tilde{\mathbf T}_i,\tilde{\mathbf N}_i\}_{i=1}^N$ of the local tangent and normal spaces,  the threshold 
$\epsilon=10^{-12}$.

\For{$i\in\{1,\cdots,N\}$}
\State Construct the matrix $\bm{\mathcal{Q}}_i\in\mathbb{R}^{n\times {\mathfrak K}}$, whose columns define $\bm{\tau}_w$ and $\mathbf{s}_w$, for $w=1,2,\dots,\mathfrak K$.

\State Construct a Vandermonde-like matrix of monomials, $\bm\Xi_i$, and collect the normal deviations in a matrix $\mathbf{S}_i$:

\[
\bm\Xi_i = \begin{bmatrix}
    \phi_1(\bm\tau_1) & \dots & \phi_Y(\bm\tau_1)
    \\
    \vdots & & \vdots
    \\
    \phi_1(\bm\tau_{\mathfrak K}) & \dots & \phi_Y(\bm\tau_{\mathfrak K})
\end{bmatrix} \in \mathbb{R}^{\mathfrak K\times Y}, \quad \mathbf{S}_i = \begin{bmatrix}
    \mathbf{s}_1^\top 
    \\
    \vdots 
    \\
    \mathbf{s}_{\mathfrak K}^\top 
\end{bmatrix} \in \mathbb{R}^{\mathfrak K\times (n-d)}.
\]

\State \parbox[t]{\dimexpr\linewidth-\algorithmicindent}{
Obtain the coefficients $\mathbf{B}_i=\{\beta_{i,j}^{(r)}\}$, with  $r=1,\dots,n-d$ and $j=1,\dots,Y$ of the intrinsic polynomial $\tilde{\pi}_i(\zeta)$

\[
\mathbf{B}_i=(\bm\Xi_i^\top\bm\Xi_i)^{-1}\bm\Xi_i^\top \mathbf{S}_i.
\]
where the $r$th column of $\mathbf{B}_i$ gives $\beta_{i,j}^{(r)}$ for $j=1,\dots,Y$.
}

\State \parbox[t]{\dimexpr\linewidth-\algorithmicindent}{%
Construct the improved approximation of the local tangent vectors with

\[\hat{\mathbf T}_i(\bm{\tau})=\tilde{\mathbf T}_i+\tilde{\mathbf N}_i\Big(D\tilde{\pi}_i(\bm{\tau})\Big).
\]

 Compute the updated normal vectors: $\hat{\mathbf N}_i = \text{null}(\hat{\mathbf{T}}_i^\top)$.}

 \State The next steps improve the estimation of $\mathbf{B}_i$.
 \While{$\|D \tilde \pi_i(0)\|_2>\epsilon$}
\State \parbox[t]{\dimexpr\linewidth-\algorithmicindent}{%
Update the tangent-normal frame with the previously computed GMLS approximation: $\tilde{\mathbf T}_i,\tilde{\mathbf N}_i \leftarrow \hat{\mathbf T}_i,\hat{\mathbf N}_i$.

Repeat steps 2 and 3.

Recompute the coefficients $\mathbf{B}_i$ in the updated tangent-normal frame

\[
\mathbf{B}_i \leftarrow (\bm\Xi_i^\top\bm\Xi_i)^{-1}\bm\Xi_i^\top \mathbf{S}_i,
\]

and construct the improved approximation as in Step 5 above with output 

\[
\hat{\mathbf T}_i(\bm{\tau})=\tilde{\mathbf T}_i+\tilde{\mathbf N}_i \Big(D\tilde \pi_i(\bm{\tau})\Big).
\]

Compute the updated normal vectors as before: $\hat{\mathbf N}_i = \text{null}(\hat{\mathbf{T}}_i^\top)$.
}

\EndWhile
\EndFor

\Ensure{Intrinsic polynomial approximation for local parametrization $\{\tilde \pi_i(\bm{\tau})\}_{i=1}^N$.}
\end{algorithmic}
\end{algorithm}

\section{Computational Complexity}\label{Appendix: complexity}

In this section, we provide details on the computational complexity our our numerical procedure (see Algorithm \ref{alg:summary}).
We assume that the prolongation coefficients $\eta_{i,J}$, defined in equations \eqref{eqn_pro} 
and \eqref{eq: prolongation formula},
have been precomputed.

We first present results for the $k$th prolongation step, i.e., approximating $\Tilde{\mathbf{z}}_i^{(k)}$ from $\Tilde{\mathbf{z}}_i^{(k-1)}$.
Our results are presented across all $N$ data points.
The first step is to approximate the tangent-normal frame via the SVD approach, which involves the following steps.
\begin{itemize}
    \item Constructing the nearest neighbor matrix using MATLAB's `knnsearch' syntax, which is at worst $\mathcal{O}(D_k N^2)$ and at best $\mathcal{O}(D_k N \log N)$, depending on the size of the data set.
    See \cite{friedman1977algorithm} for further details.
    \item Computing the distance matrix $\mathbf{D}_i \in \mathbb{R}^{D_k\times \mathfrak{K}}$ has complexity $\mathcal{O}(N D_k \mathfrak{K})$.
    \item Taking the SVD of the distance matrix has complexity $\mathcal{O}(N D_k \mathfrak{K} \min(D_k,\mathfrak{K}))$.
\end{itemize}
Therefore, the SVD step has complexity dominated by the nearest neighbor search $\mathcal{O}(D_k N^2)$.
The next step is implementing the GMLS refinement, say $\mathfrak J$ times, which involves the following steps.
\begin{itemize}
    \item Computing the projection $\mathbf{\mathcal{Q}}_i$, which is $\mathcal{O}(N \mathfrak{J} D_k \mathfrak{K})$.
    \item Forming the Vandermonde-like matrix of size $\mathfrak{K}\times Y$ for the least squares problem, which is $\mathcal{O}(d \ell Y \mathfrak{K} \mathfrak{J} N)$ ($d$ multiplications of tangent vectors, each raised to a maximum degree of $\ell$, repeated over $\mathfrak K$ rows and $Y$ columns).
    \item Solving the least squares problem \eqref{eq: least squares problem}, 
    which is $\mathcal{O}(N Y^2 \mathfrak{JK})$.
\end{itemize}
Each sub-step involving the GMLS algorithm is dominated by the $\mathcal{O}(D_k N^2)$ complexity of the nearest neighbor search.

The final step of the algorithm involves computing the SVD of $\Tilde{\mathbf{P}}\in \mathbb{R}^{K\times N(D_p-d)}$, which has complexity that depends linearly on $N$.
Therefore, the nearest neighbor search dominates the complexity of the algorithm.
Because we repeat the prolongation procedure $p$ times, we may pessimistically write that the complexity is at best $\mathcal{O}(p D_p N \log N)$, and at worst $\mathcal{O}(p D_p N^2)$.